\documentclass[11pt]{article}

\newif\ifobsolete
\obsoletefalse 

\usepackage{amsmath, amsthm, amssymb, amsfonts}
\usepackage{graphicx}
\usepackage[dvips]{epsfig}
\usepackage[usenames,dvipsnames]{color}
\usepackage{fullpage}
\usepackage{enumerate}
\usepackage{mathscinet} 
\usepackage{url} 
\usepackage{a4wide}
\usepackage{latexsym}
\usepackage{authblk}

\if10     
\usepackage[mathlines]{lineno}
\newcommand*\patchAmsMathEnvironmentForLineno[1]{%
  \expandafter\let\csname old#1\expandafter\endcsname\csname #1\endcsname
  \expandafter\let\csname oldend#1\expandafter\endcsname\csname end#1\endcsname
  \renewenvironment{#1}%
     {\linenomath\csname old#1\endcsname}%
     {\csname oldend#1\endcsname\endlinenomath}}%
\newcommand*\patchBothAmsMathEnvironmentsForLineno[1]{%
  \patchAmsMathEnvironmentForLineno{#1}%
  \patchAmsMathEnvironmentForLineno{#1*}}%
\AtBeginDocument{%
\patchBothAmsMathEnvironmentsForLineno{equation}%
\patchBothAmsMathEnvironmentsForLineno{align}%
\patchBothAmsMathEnvironmentsForLineno{flalign}%
\patchBothAmsMathEnvironmentsForLineno{alignat}%
\patchBothAmsMathEnvironmentsForLineno{gather}%
\patchBothAmsMathEnvironmentsForLineno{multline}%
}
\linenumbers
\fi

\usepackage{color}
\usepackage{hyperref} 
\definecolor{modra3}{rgb}{.1,.0,.4}
\hypersetup{colorlinks=true, linkcolor=modra3, urlcolor=modra3, citecolor=modra3, pdfpagemode=UseNone, bookmarks=true, pdfstartview=} 
\usepackage{breakurl}

\theoremstyle{plain}
\newtheorem{theorem}{Theorem}
\newtheorem{otheorem}{Obsolete theorem}
\newtheorem{claim}[theorem]{Claim}
\newtheorem{lemma}[theorem]{Lemma}
\newtheorem{olemma}[theorem]{Obsolete lemma}
\newtheorem{question}[theorem]{Question}
\newtheorem{proposition}[theorem]{Proposition}
\newtheorem{oproposition}[theorem]{Obsolete proposition}
\newtheorem{corollary}[theorem]{Corollary}
\newtheorem{observation}[theorem]{Observation}

\theoremstyle{definition}

\theoremstyle{remark}
\newtheorem*{remark}{Remark}

\DeclareMathOperator{\conv}{conv}
\DeclareMathOperator{\aff}{aff}
\DeclareMathOperator{\interior}{int}
\DeclareMathOperator{\Vol}{Vol}
\DeclareMathOperator{\sa}{sa}
\DeclareMathOperator{\msa}{msa}
\DeclareMathOperator{\supp}{supp}
\DeclareMathOperator{\dist}{dist}

\newcommand{\R}{\mathbb{R}}
\newcommand{\N}{\mathbb{N}}
\newcommand{\mbf}[1]{{\mathbf{#1}}}

\newcommand{\F}{\mathcal{F}}

\renewcommand{\aa}{\mbf{a}}
\newcommand{\bb}{\mbf{b}}

\newcommand{\hh}{\mbf{h}}

\newcommand{\pp}{\mbf{p}}
\newcommand{\rr}{\mbf{r}}

\newcommand{\uu}{\mbf{u}}
\newcommand{\vv}{\mbf{v}}
\newcommand{\ww}{\mbf{w}}
\newcommand{\xx}{\mbf{x}}
\newcommand{\yy}{\mbf{y}}
\newcommand{\zz}{\mbf{z}}
\newcommand{\OO}{\mbf{0}}

\newcommand{\CC}{C}

\renewcommand{\H}{\mathcal{H}}
\newcommand{\Q}{\mathcal{Q}}

\newcommand{\gray}[1]{{\color{gray}\bf #1}}
\definecolor{seda}{rgb}{.7,.7,.7}

\newcounter{sideremark}

\newcommand{\csup}{c^{\sup}}

\begin{document}

\title{Bounds for Pach's selection theorem and for the minimum solid angle in a
simplex}


\author[1]{Roman Karasev}
\author[2, 3, 4]{Jan Kyn\v{c}l}
\author[5]{Pavel Pat\'{a}k}
\author[2]{Zuzana Pat\'{a}kov\'{a}}
\author[2, 6]{Martin Tancer}

\affil[1]{\small Moscow Institute of Physics and Technology, Institutskiy
per.~9, Dolgoprudny,
Russia~141700 and Institute for Information Transmission Problems RAS, Bolshoy
Karetny per.~19, Moscow, Russia~127994}
\affil[2]{\small Department of Applied Mathematics and Institute for
Theoretical Computer
Science,
Charles University,
Malostransk\'e n\'am.~25, 118~00~~Praha~1, Czech Republic}
\affil[3]{\small Alfr\'ed R\'enyi Institute of Mathematics, Re\'altanoda u.
13-15, Budapest 1053, Hungary}
\affil[4]{\small \'Ecole Polytechnique F\'ed\'erale de Lausanne, Chair of
Combinatorial Geometry,
EPFL-SB-MATHGEOM-DCG, Station 8, CH-1015 Lausanne, Switzerland}
\affil[5]{\small Department of Algebra, Charles University,
Sokolovsk\'a 83, 186~75~~Praha~8, Czech Republic}
\affil[6]{\small IST Austria, Am Campus~1, 3400~Klosterneuburg, Austria
}

\maketitle

\begin{abstract}
We estimate the selection constant in the following geometric
selection theorem by Pach: For every positive integer $d$ there is a constant
$c_d > 0$ such that whenever $X_1, \dots, X_{d+1}$ are $n$-element subsets  of
$\R^d$, then we can find a point $\pp \in \R^d$ and subsets $Y_i \subseteq X_i$
for every $i \in [d+1]$, each of size at least $c_d n$, such that $\pp$ belongs
to all {\em rainbow\/} $d$-simplices determined by $Y_1, \dots, Y_{d+1}$, that
is, simplices with one vertex in each $Y_i$.

We show a super-exponentially decreasing upper bound
$c_d\leq e^{-(1/2-o(1))(d \ln d)}$.
The ideas used in the proof of the upper bound also help us prove 
Pach's theorem with $c_d \geq 2^{-2^{d^2 + O(d)}}$, which is a lower bound doubly
exponentially decreasing in $d$ (up to some polynomial in
the exponent). For comparison, Pach's original approach yields a triply
exponentially decreasing lower bound. On the other hand, Fox, Pach, and Suk recently
obtained a hypergraph density result implying a proof of Pach's theorem with
$c_d \geq2^{-O(d^2\log d)}$.

In our construction for
the upper bound, we use the fact that the minimum solid angle of every
$d$-simplex is super-exponentially small. This fact was previously unknown and might
be of independent interest.
For the lower bound, we improve the `separation' part of the argument by
showing that in one of the key steps only $d+1$ separations are necessary,
compared to $2^d$ separations in the original proof. 

We also provide a measure version of Pach's theorem.
\end{abstract}

\section{Introduction}
Selection theorems have attracted a lot of interest in discrete
geometry. We focus on the positive fraction selection theorem by
Pach~\cite{pach98}. For a more compact statement, we first introduce the following terminology. Let $S_1, \dots, S_{d+1}$ be subsets of $\R^d$. By
an \emph{$(S_1, \dots, S_{d+1})$-simplex} we mean the convex hull of points
$s_1, \dots, s_{d+1}$ where $s_i \in S_i$ for $i \in [d+1]$. Note that an
$(S_1, \dots, S_{d+1})$-simplex might be degenerate if the points $s_i$ are
  not in general position. Figure~\ref{f:pach_thm} illustrates the
  statement of the theorem.

  \begin{theorem}[Pach~\cite{pach98}]
\label{t:pach}
  For every positive integer $d$, there exists a constant $c_d > 0$ with the following
  property. Let $X_1, \dots, X_{d+1}$ be $n$-element subsets of $\R^d$. 
  Then there exist a
  point $\pp \in \R^d$ and subsets $Y_i \subseteq X_i$ for $i \in [d+1]$ each
  of them of size at least $c_d n$ such that the point
  $\pp$ belongs to all $(Y_1, \dots, Y_{d+1})$-simplices.
\end{theorem}

For a fixed $d$, we denote by $\csup_d$ the supremum of the
constants with which the theorem remains valid and we call this value
\emph{Pach's} (\emph{selection}) \emph{constant}.\footnote{Although we are
interested in the dependence of $\csup_d$ on $d$, we call it a constant emphasizing
its independence on the size of the sets $X_i$.} We do not need this
fact but it is not hard to verify that the supremum coincides with the maximum in
this case, using the finiteness of the sets $X_i$. Our aim is to estimate 
$\csup_d$. Although Pach's proof of Theorem~\ref{t:pach} is nice and
elegant, it uses several advanced tools: a weaker selection theorem, the weak
hypergraph regularity lemma, and the same-type lemma. These tools yield
a lower bound on $\csup_d$, which is roughly triply exponentially decreasing in $d$.

The goal of this paper is to establish tighter bounds on $\csup_d$. We will
show a super-exponentially decreasing upper bound
on $\csup_d$. The idea for the construction for the upper bound is relatively
straightforward. We just place the points of the sets $X_1, \dots, X_{d+1}$
uniformly in the unit ball. The analysis of this construction requires
two important ingredients. One ingredient is the analysis of the regions where 
the sets $Y_i$ from Theorem~\ref{t:pach} can appear. Using a certain separation
lemma (see Lemma~\ref{l:separate}) we can deduce that they appear in ``corner
regions'' of arrangements of $d+1$ hyperplanes. The second ingredient is an upper bound on the minimum \emph{solid} angle in a
simplex. This bound helps us to bound the sizes of the corner regions for
$Y_i$. We could not find any bound on the minimum solid angle in a simplex in
the literature. We provide a super-exponentially decreasing upper bound, which might be of independent interest.

The description of the corner regions and Lemma~\ref{l:separate} also allow us 
to obtain a doubly exponentially decreasing lower bound on $\csup_d$. More
concretely, we will show that $\csup_d \geq 2^{-2^{d^2 + O(d)}}$. Shortly
before making a preprint version of this paper publicly available, we have learned that
Fox, Pach, and Suk expected to obtain an impressive lower bound $\csup_d \geq 2^{-O(d^3\log
d)}$. Later, they improved the lower bound to $\csup_d \geq 2^{-O(d^2\log d)}$~\cite{FPS15_semialgebraic}. 

\begin{figure}
\begin{center}
  \includegraphics{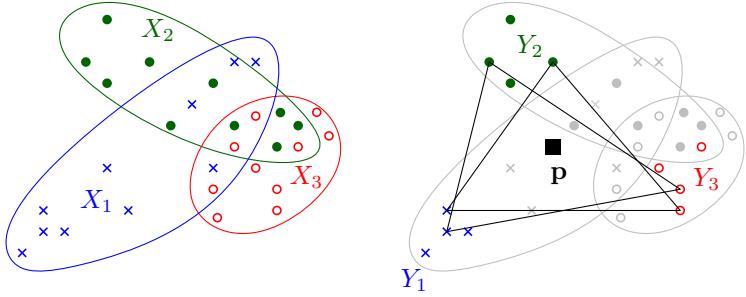}
\end{center}
\caption{Pach's theorem: initial configuration (left) and the resulting sets
$Y_i$ and the resulting point $\pp$ (right)}
\label{f:pach_thm}
\end{figure}

\begin{theorem}
\label{t:bounds}
 Pach's selection constant can be bounded as follows.
\begin{enumerate}
  \item[{\rm (1)}]
    $\csup_d \leq e^{-(1/2 - o(1)) d \ln d}$
    and
  \item[{\rm (2)}]
    $\csup_d \geq 2^{-2^{d^2 + 3d}}$.
\end{enumerate}
\end{theorem}

\ifobsolete
\gray{
\begin{otheorem}
\label{t:bounds_obsolete}
 Pach's selection constant can be bounded as follows.
\begin{enumerate}
  \item[{\rm (1)}]
    $\csup_d \leq \kappa^d$ where $\kappa < 1$; (in our proof we can achieve $\kappa < 0.9975$)
    and
  \item[{\rm (2)}]
    $\csup_d \geq 2^{-2^{d^2 + 3d}}$.
\end{enumerate}
\end{otheorem}
}
\fi
The minimum solid angle of a simplex is discussed in Section~\ref{s:msa}.
Section~\ref{s:corner} contains the description of the corner regions and the
separation lemma (Lemma~\ref{l:separate}) we need. Section~\ref{s:ub}
contains the proof of Theorem~\ref{t:bounds}(1) and Section~\ref{s:lb} contains the
proof of Theorem~\ref{t:bounds}(2).

\subsubsection*{Other selection theorems}

The following weaker selection theorem is related to the positive fraction
selection theorem of Pach. By \emph{general position} in $\R^d$ we mean that 
each set of at most $d+1$ points is affinely independent;
the general position assumption in the theorem below is not crucial but we choose the simplest statement in this
case.
\begin{theorem}
\label{t:barany}
For every $d \in \N$, there is a constant $k_d > 0$ with the following property.
  Let $P$ be a set of $n$ points in general position in $\R^d$. Then there is a
  point in at least $k_d \cdot \binom n{d+1} - O(n^d)$ $d$-simplices spanned by $P$.
\end{theorem}
Note that $\binom n{d+1}$ is the number of all $d$-simplices spanned by $P$,
thus the statement of Theorem~\ref{t:barany} says that we can indeed select a
positive fraction of simplices sharing a point.  It is not hard to see that
Theorem~\ref{t:barany} follows from Theorem~\ref{t:pach} as soon as only the
existence of $k_d$ is concerned (by splitting $P$ into $X_1, \dots, X_{d+1}$,
possibly forgetting few points).

The planar case of Theorem~\ref{t:barany} is due to Boros and
F\"{u}redi~\cite{boros-furedi84} ($d = 2$); it was extended to arbitrary
dimension by B\'ar\'any~\cite{barany82}.
B\'ar\'any proved the theorem with $k_d = \frac{1}{(d+1)^d}$.

A significant improvement to $k_d$ was found by Gromov~\cite{gromov10}
using topological methods in a much more general setting (obtaining a proof with $k_d = \frac 1{(d+1)!}$).
The first author~\cite{karasev12} found a simpler proof (still in quite general setting) and Matou\v{s}ek and
Wagner~\cite{matousek-wagner14} extracted the combinatorial essence of
Gromov's proof allowing them to get a further (slight) improvement on $k_d$.
Kr\'a\soft{l}, Mach and Sereni~\cite{kral-mach-sereni12} obtained a further
improvement of the value focusing on the combinatorial part extracted by
Matou\v{s}ek and Wagner. We do not attempt to enumerate the bounds obtained
in~\cite{matousek-wagner14,kral-mach-sereni12}.

The following variant of Theorem~\ref{t:barany} for rainbow simplices is an
important step in the proof of Theorem~\ref{t:pach}.

\begin{theorem}
\label{t:karasev}
For every $d \in \N$, there is a constant $k'_d > 0$ with the following property.
Let $X_1, \dots, X_{d+1}$ be pairwise disjoint $n$-element subsets of $\R^d$
  whose union is in general position. Then there is a point $\pp \in \R^d$ which is contained in the interior of
  at least $k'_d \cdot n^{d+1} -
  O(n^d)$ rainbow $d$-simplices, where a rainbow simplex meets each $X_i$ in
  exactly one vertex and $k'_d > 0$ is a constant depending only on $d$.
\end{theorem}

Theorem~\ref{t:karasev} is implicitly proved in~\cite{pach98} with $k'_d$
roughly around $\frac1{(5d)^{d^2}}$. The proof in~\cite{karasev12}
(following Gromov) gives the result with $k'_d = \frac1{(d+1)!}$. The constant has been recently improved to $k'_d = \frac{2d}{(d+1)!(d+1)}$~\cite{Jia14_slight}. We note that
the main result in~\cite{karasev12} and~\cite{Jia14_slight} is in the setting of absolutely continuous measures.
It can be easily transformed into the setting of Theorem~\ref{t:karasev} by
replacing each point $x \in X_1 \cup \cdots \cup X_{d+1}$ by a sufficiently
small ball centered in $x$ and using the 
fact that for a sufficiently small $\varepsilon$, any point of $\R^d$ can be
$\varepsilon$-close to the boundary of at most $O(n^d)$ simplices spanned by
$X_1 \cup \cdots \cup X_{d+1}$. This follows from the fact that every point of $\R^d$ is in at most $O(n^{d-1})$ hyperplanes spanned by $X_1 \cup \cdots \cup X_{d+1}$~\cite[Lemma 9.1.2]{matousek02}.

An interesting selection theorem in a `dual' setting was recently 
obtained by B\'{a}r\'{a}ny and Pach~\cite{barany-pach14}. 
A variant of Pach's theorem for hypergraphs with bounded degree was, also
recently, obtained by Fox et al.~\cite{fox-gromov-lafforgue-naor-pach12}.

\subsubsection*{Measure version of Pach's theorem}

Due to the similarity of Pach's theorem to other geometric selection theorems,
such as Theorem~\ref{t:karasev}, one can expect that Pach's theorem also admits
a measure version, where point sets are replaced with probability measures.
We will indeed verify this expectation (with the same value for the selection
constant). We prove the theorem for Borel probability measures, which generalize both finite point sets and bounded absolutely continuous measures.

We recall that $\mu$ is a \emph{Borel probability measure} on $\R^d$ if $\mu$ is a nonnegative measure defined on the $\sigma$-algebra of Borel subsets of $\R^d$ and $\mu(\R^d)=1$.

\begin{theorem}
\label{t:pach_continuous}
 Let $\mu_1,\ldots, \mu_{d+1}$ be Borel probability measures on $\R^d$. Then there exist sets $Z_i\subseteq \R^d$ with
 $\mu_i(Z_i)\geq 2^{-2^{d^2+3d}}$ and a point $\pp\in \R^d$ contained in all $(Z_1,\ldots, Z_{d+1})$-simplices.
\end{theorem}

Theorem~\ref{t:pach_continuous} follows from Theorem~\ref{t:bounds}(2) by approximating Borel measures as weak limits of discrete measures. The reduction relies on the fact that each of the sets $Y_i$ in Pach's theorem can be obtained as an intersection of $X_i$ with a simplicial cone, a region of small ``geometric complexity''.
We prove Theorem~\ref{t:pach_continuous} in Section~\ref{section_measure}. 

\section{The minimum solid angle in a simplex}
\label{s:msa}

We start our preparations for the proof of Theorem~\ref{t:bounds}(1) by
bounding the minimum solid angle in a simplex. 

Let $\Delta$ be a $d$-simplex and $\vv$ be a vertex of $\Delta$. By the \emph{solid
angle} at $\vv$ in $\Delta$ we mean the value
\[
\sa(\vv; \Delta) \mathrel{\mathop:}= \frac{\Vol( B(\vv; \varepsilon) \cap \Delta)}{\Vol(B(\vv; \varepsilon))}
\]
 where $B(\xx; r)$ denotes the ball centered in $\xx$ with radius $r$;
 $\varepsilon$ is small enough (so that $B(\vv; \varepsilon)$ does not meet the
 hyperplane determined by the vertices of $\Delta$ except $\vv$); and $\Vol$
 denotes the $d$-dimensional volume (that is, the $d$-dimensional Lebesgue measure).
 Note that in our case the solid angle is normalized, that is, it measures the
 probability that a random point of $B(\vv; \varepsilon)$ belongs to the
 simplex. Note also that the solid angle can be equivalently defined as the ratio of the
$(d-1)$-dimensional volume of the spherical simplex $\partial B(\vv;
\varepsilon) \cap \Delta$ and the $(d-1)$-dimensional volume of the sphere
$\partial B(\vv; \varepsilon)$. For our needs, however, the definition via $d$-volumes is much more convenient.

 Our goal is to give the upper
 bound on the minimum solid angle of $\Delta$: 
\[
\msa(\Delta) \mathrel{\mathop:}= \min\{\sa(\vv; \Delta)\colon \hbox{$\vv$ is a vertex of
$\Delta$}\}.
\]
\begin{theorem}
  \label{t:msa}
  The minimum solid angle of any $d$-simplex $\Delta$ satisfies $\msa(\Delta)
  \leq e^{-(1/2 - o(1))(d \ln d)}$.
\end{theorem}

Before we prove Theorem~\ref{t:msa}, let us remark that in general we
consider determining the upper bound on $\msa(\Delta)$ as an interesting
question. Let $\rho_d$ be the solid angle in the regular $d$-simplex.
Obviously any upper bound on $\msa(\Delta)$ for a $d$-simplex $\Delta$ is at
least $\rho_d$. On the other hand, we are not aware of any example of a
$d$-simplex $\Delta$ with $\msa(\Delta) > \rho_d$. Thus, we suggest the
following question.

\begin{question}
  \label{q:msa}
  Is it true that $\msa(\Delta) \leq \rho_d$ for any $d$-simplex $\Delta$?
If the answer is negative, what is the least upper bound on $\msa(\Delta)$ and
for which simplex is it attained?
\end{question}

Akopyan and the first author show~\cite{akopyan-karasev15arxiv} that the answer is affirmative if $d \leq 4$.

Rogers~\cite{Rogers61_asymptotic} derived an asymptotic formula for the surface area of a regular spherical simplex, which implies the following asymptotic formula for $\rho_d$:
\[
\rho_d=\frac{\sqrt{d+1}}{\sqrt{2}e2^d} \cdot \left(\frac{2e}{\pi d}\right)^{d/2} \cdot \Bigl(1+O\Bigl(\frac{1}{d}\Bigr)\Bigr).
\]
Further asymptotic simplification gives $\rho_d = e^{-(1/2 + o(1))(d\ln d)}$. This
shows that our bound in Theorem~\ref{t:msa} is tight up to lower order terms in the
exponent. Rogers' proof is also reproduced in a book by Zong~\cite[Lemma
7.2]{Zong99_Sphere}. We have learnt about this from an answer of
Joseph O'Rourke~\cite{Rourke11_MO} to a question of Boris Bukh at MathOverflow.

 The simplified asymptotic formula for $\rho_d$, up to lower order terms in the
 exponent, also follows by the following easy approximation. Let $\Delta = \Delta_1$ be a regular unit $d$-simplex, let $\vv$ be a vertex of
 $\Delta$ and let $\Delta_\kappa$ be a homothetic copy of $\Delta$ under a
 homothety centered at $\vv$ with coefficient $\kappa > 0$. Simple computation
 shows that the length of the median in $\Delta$ is at least $1/\sqrt 2$, and therefore $\Delta_\varepsilon \subseteq B(\vv; \varepsilon) \cap \Delta \subseteq \Delta_{\sqrt 2 \varepsilon}$.
 This gives 
 \[
 \Vol(\Delta_\varepsilon)  \le  \varepsilon^d \rho_d \beta_d  \le   \Vol(\Delta_{\sqrt 2\varepsilon})
\]
 where $\beta_d=\frac{\pi^{d/2}}{\Gamma(d/2+1)}$ is the volume of the unit $d$-ball.
 Using that 
\[
\Vol(\Delta_\kappa) =\kappa^d \frac{\sqrt{d+1}}{d!2^{d/2}}
\]
and the estimates $\Gamma(d/2+1)=e^{(1/2-o(1))\cdot d\ln d}$ and
$d!=e^{(1-o(1))\cdot d\ln d}$, we obtain that $\rho_d = e^{-(1/2\pm o(1))\cdot d \ln d}$.


\subsubsection*{Normal cones and spherical Blaschke--Santal\'{o} inequality}

Now, we focus on a proof of Theorem~\ref{t:msa}. The main step is to use the
Spherical Blaschke--Santal\'{o} inequality, which allows us to bound the solid
angle (of a cone) if we know the solid angle of the polar cone.
The idea with polar cones was suggested by Yoav Kallus~\cite{kallus13_MO}. In a 
previous version of this paper we obtained Theorem~\ref{t:msa} with a weaker, 
exponentially decreasing, bound with a self-contained
proof~\cite{kyncl-patak-safernova-tancer14arxiv-v2}. 
Later we found the current proof using the spherical Blaschke--Santal\'{o} inequality.

We start with a few definitions
and known results.
Let $C \subseteq \R^d$ be a closed convex cone with apex in the origin. 
By the (restricted) \emph{volume} of the cone $C$ we mean the the value $\Vol'(C) \mathrel{\mathop:}=
\Vol(C \cap B^d)$, where $B^d$ is the unit ball centered in the origin.
The \emph{polar} (or
\emph{normal}) cone to $C$ is the cone 
$$
C^* = \{\xx \in \R^d \colon \xx \cdot \yy \leq 0 \hbox{ for any } \yy \in C\}. 
$$
A closed convex cone $C$ with apex in the origin is \emph{round} if the 
intersection $\partial C \cap \partial B^d$ is a (geometric) $(d-2)$-sphere. We need the 
following theorem which relates the (restricted) volumes of $C$ and $C^*$. By $\beta_d$ we 
denote the volume of $B^d$.

\begin{theorem}[Spherical Blaschke--Santal\'{o}
  inequality~{\cite[eq. (21)]{gao-hug-schneider03}}]
\label{t:bs}

  Let $w \in (0, \frac12 \beta_d)$ be a fixed number. Let $C$ be a closed
  convex cone
  with apex in the origin such that $\Vol'(C) = w$. Then $\Vol'(C^*)$ is maximal
  if $C$ is a round cone.

\end{theorem}

Note that Theorem~\ref{t:bs} is stated in~\cite{gao-hug-schneider03} in the
setting of spherical $(d-1)$-volumes of $C\cap \partial B^d$. However, our
small change in the setting does not affect the extremal property.

Given a $d$-simplex $\Delta$ with vertices $\vv_1, \dots, \vv_{d+1}$ and $i \in [d+1]$, let $C_i$ be the cone with apex in the origin obtained by shifting the cone with apex $\vv_i$ determined by $\Delta$. Then the spherical angle $\sa(\vv_i, \Delta)$ can be expressed as
$\Vol'(C_i)/\beta_d$.
An important well-known observation is that the polar cones $C^*_i$ cover the
space (they form so-called \emph{normal fan}).

\begin{lemma}
\label{l:polars_cover}
  The cones $C^*_i$ cover $\R^d$. Consequently, there is $i \in [d+1]$ such
  that $\Vol'(C^*_i) \geq \frac{1}{d+1} \beta_d$.
\end{lemma}

\begin{proof} For completeness, we sketch a proof. Let $\xx \in \R^d$ and let $i
  \in [d+1]$ be such that $\xx \cdot \vv_i$ is maximal among all choices
  of $i$. Then $\xx \cdot (\yy - \vv_i) \leq 0$ for any $\yy \in \Delta$ which
  implies that $\xx \in C^*_i$.
\end{proof}

By Lemma~\ref{l:polars_cover}, there is a polar cone $C^*_i$ with large volume. By
 Blaschke--Santal\'{o} inequality, the cone $C_i$ must have small volume. Using the concentration of the measure on the sphere,
we estimate $\Vol'(C_i)$ from above.
 We present an elementary argument, since we do not need the concentration of the measure in its full strength.

\begin{lemma}
\label{l:round_polar}
  Let $C^* \subseteq \R^d$ be a round cone such that $\Vol'(C^*) \geq
  \frac1{d+1}\beta_d$. Then $\Vol'(C) \leq
  e^{-(1/2 - o(1))(d \ln d)}\beta_d$.
\end{lemma}

\begin{proof}
Without loss of generality, we assume that the $x$-axis (that is, the first-coordinate axis in $\R^d$) is the axis of symmetry
of $C^*$. Let $h$ be the hyperplane determined by the $(d-2)$-sphere $\partial
C^* \cap \partial B^d$. Let $\gamma$ be the distance of $h$ from the origin.
Since $\Vol'(C^*) \geq \frac1{d+1}\beta_d$, we deduce that $\gamma \leq \frac12$ and
therefore $C^* \cap B^d$ fits into a ball of radius $\sqrt{1-\gamma^2}$, centered in the
intersection of $h$ and the $x$-axis; see Figure~\ref{f:volume_dual_cone}
(left). (We have borrowed this idea from~\cite{racke06}, aiming at a reasonable
estimate without precise computation.) Consequently, $\Vol'(C^*) \leq (1 -
\gamma^2)^{d/2} \beta_d$, which implies
\begin{equation}
\label{e:gamma}
  \gamma^2 \leq 1 - \left(\frac1{d+1}\right)^{2/d} \leq \frac2d\ln(d+1),
\end{equation}
using the estimate $1 - x \leq -\ln x$.

\begin{figure}
\begin{center}
\includegraphics{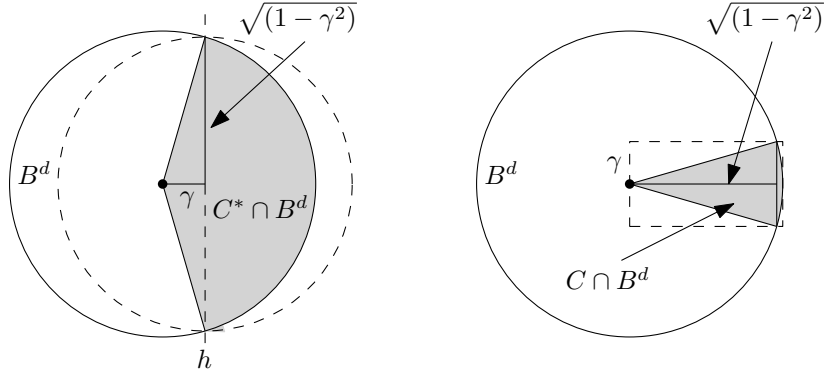}
\caption{$C^* \cap B^d$ fits into a dashed ball of radius $\sqrt{1-\gamma^2}$
whereas $C \cap B^d$ fits into a dashed cylinder $[0,1] \times B^{d-1}_\gamma$.}
\label{f:volume_dual_cone}
\end{center}
\end{figure}

On the other hand, $C$ fits into the cylinder $[0,1] \times
B^{d-1}_\gamma$ where we temporarily consider $\R^d$ as the product $\R
\times \R^{d-1}$ and $B^{d-1}_\gamma \subset \R^{d-1}$ is the ball with radius
$\gamma$ centered in the origin; see Figure~\ref{f:volume_dual_cone} (right). Therefore, 
using~\eqref{e:gamma} and $\beta_{d-1} \leq \beta_{d-2} = \frac{d}{2\pi} \beta_d$ for $d \geq 2$, 
we get
$$
\Vol'(C) \leq \gamma^{d-1}\beta_{d-1} \leq
\left(\frac2d\ln(d+1)\right)^{\frac{d-1}2}\frac{d}{2\pi} \beta_d \leq e^{-(1/2
  - o(1))(d \ln
d)}\beta_d.
$$
\end{proof}

\begin{proof}[Proof of Theorem~\ref{t:msa}]
  By Lemma~\ref{l:polars_cover} we know that $\Vol'(C^*_i) \geq
  \frac1{d+1}\beta_{d}$ for some $i \in [d+1]$. Let $C^*$ be the round cone
  such that $\Vol'(C^*_i) = \Vol'(C^*)$. By Theorem~\ref{t:bs} (for $C^*$) we know 
  that $\Vol'(C_i) \leq \Vol'(C)$ and Lemma~\ref{l:round_polar} implies that
  $\Vol'(C) \leq e^{-(1/2 - o(1))(d \ln d)} \beta_d$. Consequently,
  $$\msa(\Delta) \leq \frac{\Vol'(C_i)}{\beta_d} \leq e^{-(1/2 - o(1))(d \ln d)}$$
  as required.
\end{proof}

\ifobsolete
\gray{
\section{The minimum solid angle in a simplex (obsolete)}
\label{s:msa_obsolete}

We start our preparations for the proof of Theorem~\ref{t:bounds}(1) by
bounding the minimum solid angle in a simplex. 

Let $\Delta$ be a $d$-simplex and $v$ be a vertex of $\Delta$. By the \emph{solid
angle} at $\vv$ in $\Delta$ we mean the value
\[
\sa(\vv; \Delta) \mathrel{\mathop:}= \frac{\Vol( B(\vv; \varepsilon) \cap \Delta)}{\Vol(B(\vv; \varepsilon))}
\]
 where $B(\xx; r)$ denotes the ball centered in $\xx$ with radius $r$;
 $\varepsilon$ is small enough (so that $B(\vv; \varepsilon)$ does not meet the
 hyperplane determined by the vertices of $\Delta$ except $\vv$); and $\Vol$
 denotes the $d$-dimensional volume (that is, the $d$-dimensional Lebesgue measure).
 Note that in our case the solid angle is normalized, that is, it measures the
 probability that a random point of $B(\vv; \varepsilon)$ belongs to the
 simplex. Note also that the solid angle can be equivalently defined as the ratio of the
$(d-1)$-dimensional volume of the spherical simplex $\partial B(\vv;
\varepsilon) \cap \Delta$ and the $(d-1)$-dimensional volume of the sphere
$\partial B(\vv; \varepsilon)$. For our needs, however, the definition via $d$-volumes is much more convenient.

 Our goal is to give the upper
 bound on the minimum solid angle of $\Delta$: 
\[
\msa(\Delta) \mathrel{\mathop:}= \min\{\sa(\vv; \Delta)\colon \hbox{$\vv$ is a vertex of
$\Delta$}\}.
\]

\begin{theorem}
  \label{t:msa_obsolete}
  The minimum solid angle of an arbitrary $d$-simplex $\Delta$ satisfies $\msa(\Delta)
  \leq \gamma^d$ where $\gamma < 1$ is independent of $d$.
\end{theorem}

We prove the theorem with $\gamma = \sqrt \frac{\sqrt 3} 2 \leq 0.931$.

Before we prove Theorem~\ref{t:msa}, let us remark that in general we
  consider determining the upper bound on $\msa(\Delta)$ as an interesting
  question. Let $\rho_d$ be the solid angle in the regular $d$-simplex.
  Obviously any upper bound on $\msa(\Delta)$ for a $d$-simplex $\Delta$ is at
  least $\rho_d$. On the other hand, we are not aware of any example of a
  $d$-simplex $\Delta$ with $\msa(\Delta) > \rho_d$. Thus, we suggest the
  following question.

\begin{question}
  \label{q:msa}
  Is it true that $\msa(\Delta) \leq \rho_d$ for any $d$-simplex $\Delta$?
If the answer is negative, what is the least upper bound on $\msa(\Delta)$?
\end{question}

Rogers~\cite{Rogers61_asymptotic} derived an asymptotic formula for the surface area of a regular spherical simplex, which implies the following asymptotic formula for $\rho_d$:
\[
\rho_d=\frac{\sqrt{d+1}}{\sqrt{2}e2^d} \cdot \left(\frac{2e}{\pi d}\right)^{d/2} \cdot \Bigl(1+O\Bigl(\frac{1}{d}\Bigr)\Bigr).
\]
Further asymptotic simplification gives $\rho_d = 2^{-\Theta(d\log d)}$.
Rogers' proof is also reproduced in a book by Zong~\cite[Lemma
7.2]{Zong99_Sphere}. We have learnt about this from an answer of
Joseph O'Rourke~\cite{Rourke11_MO} to a question of Boris Bukh at MathOverflow.

 The simplified asymptotic formula for $\rho_d$ also follows by the following easy approximation. Let $\Delta = \Delta_1$ be a regular unit $d$-simplex, let $\vv$ be a vertex of
 $\Delta$ and let $\Delta_\kappa$ be a homothetic copy of $\Delta$ under a
 homothety centered at $\vv$ with coefficient $\kappa > 0$. Simple computation
 shows that the length of the median in $\Delta$ is at least $1/\sqrt 2$, and therefore $\Delta_\varepsilon \subseteq B(\vv; \varepsilon) \cap \Delta \subseteq \Delta_{\sqrt 2 \varepsilon}$.
 This gives 
 \[
 \Vol(\Delta_\varepsilon)  \le  \varepsilon^d \rho_d \beta_d  \le   \Vol(\Delta_{\sqrt 2\varepsilon})
\]
 where $\beta_d=\frac{\pi^{d/2}}{\Gamma(d/2+1)}$ is the volume of the unit $d$-ball.
 Using that 
\[
\Vol(\Delta_\kappa) =\kappa^d \frac{\sqrt{d+1}}{d!2^{d/2}}
\]
and the estimates $\Gamma(d/2+1)=2^{(1/2+o(1))\cdot d\log d}$ and $d!=2^{(1+o(1))\cdot d\log d}$, we obtain $\rho_d = 2^{-(1/2+o(1))\cdot d \log d}$.

 If it turned out that the answer to Question~\ref{q:msa} is affirmative, our
 approach (even simplified) for the proof of Theorem~\ref{t:bounds}(1) would
 yield a stronger bound 
 $\csup_d \leq 2^{-\alpha d \cdot \log d}$ for some $\alpha > 0$.

\begin{proof}[Proof of Theorem~\ref{t:msa}]
Let $\aa\bb$ be one of the longest edges of $\Delta$. Without loss of generality, at least
half of the remaining vertices of $\Delta$ are not farther from $\bb$ than from
$\aa$.
Let $\vv_1, \dots, \vv_k$ be such vertices ($k \geq (d-1)/2$) and let $\uu_1, \dots,
\uu_{\ell}$ be the remaining vertices, which are closer to $\aa$ than to $\bb$. We
observe that the angles $\vv_i\aa\bb$ are at most $60^\circ$ since
$\vv_i\bb$ is one of the shortest edges in the triangle $\aa\bb\vv_i$. We also
observe that the angles $\uu_i\aa\bb$ are at most $90^\circ$ since $\bb\uu_i$ is
at most as long as $\aa\bb$.

Let $h$ be the hyperplane perpendicular to $\aa\bb$ passing through $\aa$ and let
$h^+$ be the closed halfspace bounded by the hyperplane $h$ and containing $\bb$. 
Let $C$ be the cone with apex $\aa$ determined by $\Delta$. Fix a sufficiently small $\varepsilon>0$ such that the ball $B(\aa,\varepsilon)$ with center $\aa$ and radius $\varepsilon$ does not meet the
 hyperplane determined by the vertices of $\Delta$ other than $\vv$.
We need to determine what fraction of the ball $B(\aa,\varepsilon)$ belongs to $C$. Since all the angles $\vv_i\aa\bb$ and $\uu_i\aa\bb$ are at most $90^\circ$, it follows that $C$ is fully contained in $h^+$. 

Let $\kappa$ be the affine $(k+1)$-space determined by the face $\aa\bb\vv_1\dots \vv_k$ of $C$. 
Let $C^{\kappa}_{60}$ be the $(k+1)$-dimensional cone formed by all points $\xx$ in $\kappa$
such that the angle $\xx\aa\bb$ is at most $60^\circ$. From the discussion above it
follows that $C^\kappa_{60}$ contains all the vertices $\vv_i$, and consequently $C \cap \kappa \subseteq
C^\kappa_{60}$.
It is not too difficult to show that
\begin{equation}
\label{e:C60}
\frac{\Vol_{k+1}(B(\aa,\varepsilon) \cap C^\kappa_{60})}
{\Vol_{k+1}(B(\aa,\varepsilon)\cap \kappa)} 
\leq \left(\frac{\sqrt 3}2\right)^{k+1},
\end{equation}
where $\Vol_{k+1}$ is the $(k+1)$-dimensional volume
in $\kappa$. Indeed, the set $B(\aa,\varepsilon) \cap C^\kappa_{60}$ is contained in the
$(k+1)$-dimensional ball inside $\kappa$ of radius $\frac{\sqrt 3}2\varepsilon$; see the dashed circle in Figure~\ref{f:solid_kappa}, left.
(We have borrowed this idea from~\cite{racke06}, aiming at a reasonable
estimate without precise computation.)

\begin{figure}
\begin{center}
  \includegraphics{solid_kappa.eps}
\end{center}
\caption{Cuts through $\kappa$ and $\kappa'$}
\label{f:solid_kappa}
\end{figure}

Now we estimate $\Vol(B(\aa, \varepsilon) \cap C)/ \Vol(B(\aa, \varepsilon))$. Let
$\kappa'$ be an arbitrary $(k+1)$-space parallel to $\kappa$. 
Our goal is to show that
\begin{equation}
\label{e:vol_cut}
\frac{\Vol_{k+1}(B(\aa, \varepsilon) \cap C \cap
\kappa')}{\Vol_{k+1}(B(\aa, \varepsilon) \cap \kappa')} \leq
\left(\frac{\sqrt 3}2\right)^{k+1}.
\end{equation}
As soon as we show~\eqref{e:vol_cut} we get the same bound on $\Vol(B(\aa, \varepsilon)
\cap C)/ \Vol(B(\aa, \varepsilon))$ by the Fubini theorem.

For the rest of the proof we assume for simplicity that $\aa$ coincides
with the origin. 
In order to show~\eqref{e:vol_cut}, we first observe that $C \cap \kappa'$ is either
empty or it equals $(C \cap \kappa) + \yy$, where $\yy$ is the intersection
point of $\kappa'$ and the $\ell$-dimensional cone with apex $\aa$ determined by the points
$\uu_1, \dots, \uu_\ell$.
Thus,
in particular, $\yy \in h^+$ and $C \cap \kappa' \subseteq \yy +
C^\kappa_{60}$. See Figure~\ref{f:solid_kappa}, right.

The next step is to show that $\Vol_{k+1}(B(\aa, \varepsilon) \cap (\yy +
C^\kappa_{60})) \leq \Vol_{k+1}(B(\aa, \varepsilon) \cap (\zz + C^\kappa_{60}))$ where $\zz$
is the center of the $(k+1)$-dimensional ball $B(\aa, \varepsilon) \cap \kappa'$. See
Figure~\ref{f:solid_ineq}, left. We first shift $\yy$ to its orthogonal projection
$\yy'$ on $h$, observing that $\Vol_{k+1}(B(\aa, \varepsilon) \cap (\yy +
C^\kappa_{60})) \leq \Vol_{k+1}(B(\aa, \varepsilon) \cap (\yy' + C^\kappa_{60}))$. Then we bound $\Vol_{k+1}(B(\aa, \varepsilon) \cap (\yy' +
C^\kappa_{60}))$ by decomposing $B(\aa, \varepsilon) \cap (\yy' + C^\kappa_{60})$ into two
parts as in the middle and the right part of Figure~\ref{f:solid_ineq}.

\begin{figure}
\begin{center}
  \includegraphics{solid_ineq.eps}
\end{center}
\caption{Bounding the volume of $B(\aa, \varepsilon) \cap (\yy +
C^\kappa_{60})$}
\label{f:solid_ineq}
\end{figure}

Analogously to~\eqref{e:C60}, we have $\Vol_{k+1}(B(\aa, \varepsilon) \cap (\zz +
C^\kappa_{60}))/\Vol_{k+1}(B(\aa, \varepsilon) \cap \kappa') \leq \left(\frac{\sqrt
3}2\right)^{k+1}$. This finishes the proof of~\eqref{e:vol_cut}. Then we get the final bound
\[
\msa(\Delta) \le \left(\frac{\sqrt 3}2\right)^{(d+1)/2} \le (0.931)^d.
\] 
\end{proof}
}
\fi

\section{Corner regions}
\label{s:corner}
In this section we describe a geometric structure we are essentially looking for
in order to prove Theorem~\ref{t:bounds}.

Although Theorem~\ref{t:pach} does not assume any kind of general position,
we will need general position in our intermediate steps.
We work with arrangements of $d + 1$ hyperplanes. We say that such an
arrangement is in \emph{general position} if the normal vectors of arbitrary $d$
hyperplanes from the arrangement are linearly independent 
(in particular each $d$ of the hyperplanes have a single point in common) 
and if the intersection of all $d + 1$ of the hyperplanes is empty.

Let $\H = (H_1, \dots, H_{d+1})$ be an arrangement of hyperplanes in $\R^d$ in
general position.
For $i \in
[d+1]$ let
$\hh_i$ denote the intersection point of all hyperplanes from $\H$ but $H_i$.
It is easy to see that the arrangement $\H$ has exactly one bounded
component, namely the simplex with vertices $\hh_i$. We denote this
simplex by $\Delta(\H)$. We also denote by $H_i^+$ and $H_i^-$ the two closed
subspaces determined by $H_i$ in such a way that $H_i^-$ contains $\Delta(\H)$.
Finally, we define the \emph{corner regions} $C_i = C_i(\H)$ by setting
\[
C_i \mathrel{\mathop:}= \bigcap\limits_{j \in [d+1]\setminus \{i\}} H_j^+.
\] 
Note that each $C_i$ is
a cone with apex $\hh_i$; see Figure~\ref{f:cr}, left.

\begin{figure}
\begin{center}
  \includegraphics{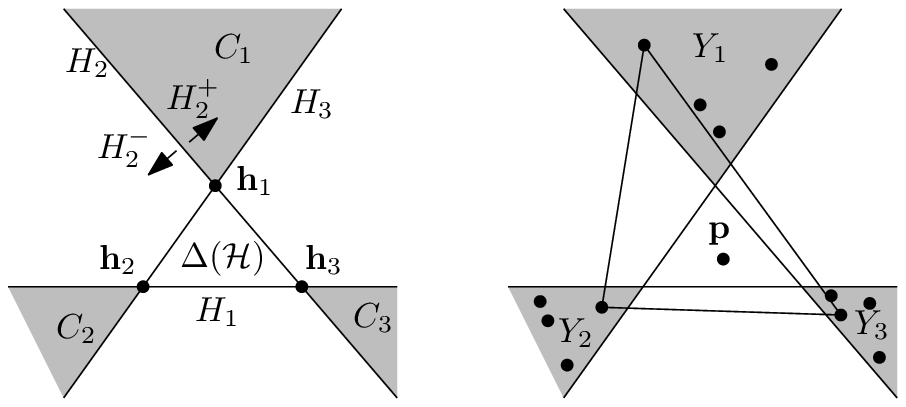}
\end{center}
\caption{Corner regions of an arrangement of $d + 1$ hyperplanes}
\label{f:cr}
\end{figure}

The following separation lemma captures the core idea of our approach. Given a
$k$-tuple $(S_1, \dots, S_k)$ of subsets of $\R^d$, by $\widehat S_i$ we mean the set $S_1
\cup \cdots \cup S_{i-1} \cup S_{i+1} \cup \cdots \cup S_k$, for any $i \in
[k]$. The interior of a set $S \subseteq \R^d$ is denoted by $\interior (S)$.

\begin{lemma}
  \label{l:separate}
  Let $\pp$ be a point in $\R^d$, let $\H$ be an arrangement of $d+1$ hyperplanes in
  general position in $\R^d$ and let $Y_1, \dots, Y_{d+1}$ be finite nonempty subsets of $\R^d$
 such that $H_i$ strictly separates $\pp$ from $\widehat Y_i$ for every $i \in [d+1]$ (in particular, $\pp$ does not belong to
  any $H_i$). Then either
\begin{itemize}
  \item[$\bullet$] $\pp \in \Delta(\H)$ and $Y_i \subseteq \interior (C_i)$ for any $i \in
    [d+1]$; or
  \item[$\bullet$] $\pp \not\in \Delta(\H)$ and there is a hyperplane strictly separating $\pp$
    from $Y_1 \cup \cdots \cup Y_{d+1}$; see Figure~\ref{f:sl}.
\end{itemize}
\end{lemma}

\begin{figure}
\begin{center}
  \includegraphics{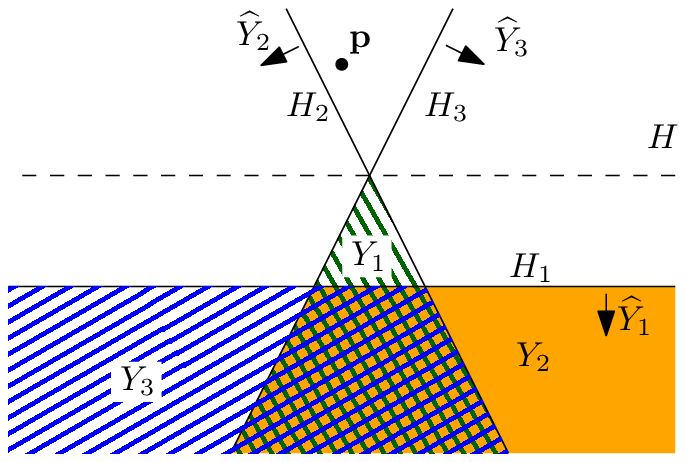}
\end{center}
\caption{Illustration for Lemma~\ref{l:separate}. In this case, the point $\pp$ belongs to
  $H_1^- \cap H_2^+ \cap H_3^+$. 
The regions of $\H$ where the sets $Y_1$, $Y_2$ and $Y_3$ may appear are shaded or striped}
\label{f:sl}
\end{figure}

\begin{proof}
Without loss of generality, we assume that $\mathbf{p}=\mathbf{0}$. For $i\in [d+1]$, let $\mathbf{u}_i$ be the unit vector normal to $H_i$ so that $H_i=\{\mathbf{x}\in \mathbb{R}^d; \, \mathbf{x}\cdot \mathbf{u}_i=c_i\}$ for some $c_i>0$. Let $K=\conv\{\mathbf{u}_1,\mathbf{u}_2,\dots,\mathbf{u}_{d+1}\}$.

By the hyperplane separation theorem for $\mathbf{p}$ and $K$, either $\mathbf{p}\in K$ or $\mathbf{p}$ is strictly separated from $K$ by a hyperplane.

If $\mathbf{p}\in K$, then $\mathbf{p}$ is in the interior of $K$, since the hyperplanes $H_i$ are in general position.
It follows that the intersection of the half-spaces $\{\mathbf{x}\in \mathbb{R}^d; \, \mathbf{x}\cdot \mathbf{u}_i\le c_i\}$ is bounded, thus $\pp$ belongs to $\Delta(\H)$ according to our definitions. Given $i, j \in [d+1]$ such that $i \neq j$,
 we get $Y_i \subseteq \interior (H_j^+)$ since $H_j$ separates $\pp$ and $\widehat Y_j$. For every fixed $i$, the previous inclusions imply that $Y_i \subseteq \interior (C_i)$.

Now suppose that $\mathbf{p}$ is strictly separated from $K$ by a hyperplane $H$.
For $i\in [d+1]$, let $Z_i\mathrel{\mathop:}= \{\mathbf{x}\in \mathbb{R}^d; \, \mathbf{x}\cdot \mathbf{y}_i>0 \hbox{ for all }
\mathbf{y}_i \in Y_i\}$. The set $Z_i$ is an open convex cone and consists of all vectors $\mathbf{x}$ such that for some $c'>0$, the hyperplane $\{\mathbf{y}\in \mathbb{R}^d; \, \mathbf{y}\cdot \mathbf{x}=c'\}$ strictly separates $\mathbf{p}$ from $Y_i$. By the assumption, every $d$-tuple of the cones $Z_i$ contains a common point in $K$, and thus it also contains a common point in $H$. By Helly's theorem for the intersections $Z_i\cap H$, we conclude that $H\cap Z_1 \cap Z_2 \cap \dots \cap Z_{d+1}$ is nonempty, and the lemma follows.
%
\end{proof}

For the proof of part (2) of Theorem~\ref{t:bounds} we need to verify an intuitively
obvious fact, that if we pick one point in each of the corner regions of an
arrangement $\H$ of $d+1$ hyperplanes, then the simplex formed by these points
covers $\Delta(\H)$.

This fact is not needed for part (1) of Theorem~\ref{t:bounds}.

\begin{lemma}
  \label{l:in_simplex}
 Let $\H = (H_1, \dots, H_{d+1})$ be an arrangement of $d+1$ hyperplanes in
 general position in $\R^d$. Let $\pp$ be a point in $\Delta(\H)$ and $\yy_1, \dots
 \yy_{d+1}$ be points in $\R^d$ such that $\yy_i \in C_i$ for any $i \in
 [d+1]$. Then $\pp$ belongs to the simplex determined by $\yy_1, \dots,
 \yy_{d+1}$.
\end{lemma}

\begin{proof}
We prove the lemma by induction on $d$. For $d=1$ the proof is obvious.
Now assume that $d > 1$.

\begin{figure}
\begin{center}
  \includegraphics{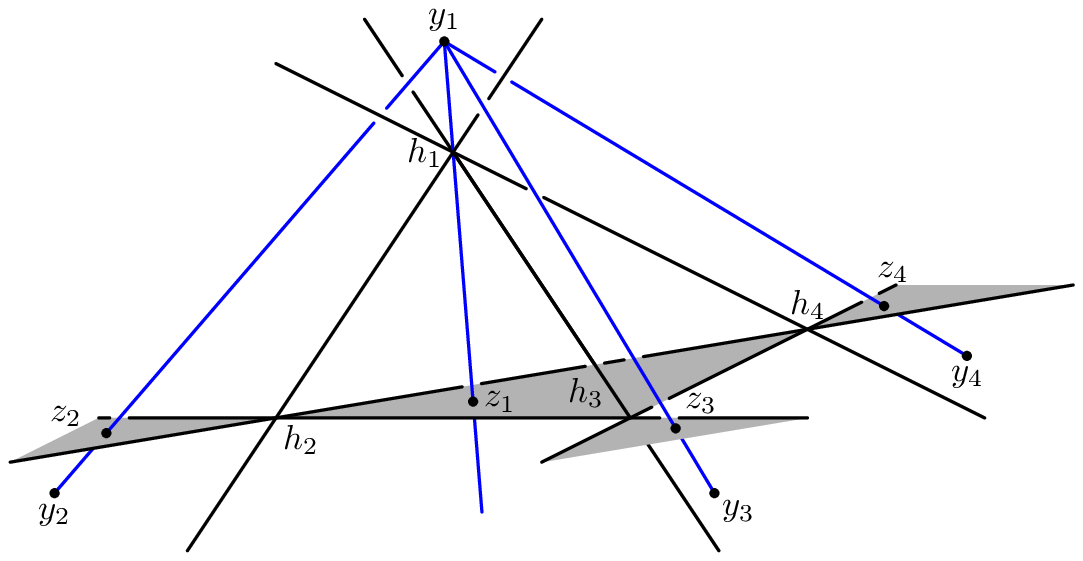}
\end{center}
\caption{The case $d=3$}
\label{f:inside}
\end{figure}

 We recall that each $C_i$ is a cone with apex $\hh_i$. Since $\pp$ is a convex
  combination of the points $\hh_i$, it is sufficient to show that
each $\hh_i$ belongs to the simplex determined by $\yy_1, \ldots, \yy_{d+1}$.

We fix $i$ and consider points
$\zz_i \mathrel{\mathop:}= \overline{\hh_i\yy_i} \cap H_i$ and $\zz_j \mathrel{\mathop:}= \overline{\yy_i\yy_j} \cap H_i$ for $j \in [d+1]\setminus \{i\}$,
where $\overline{\aa\bb}$ is the line spanned by points $\aa$
and $\bb$. See Figure \ref{f:inside}.

We claim that $\zz_j$, for $j \neq i$, belongs to $C_j'\mathrel{\mathop:}= C_j \cap H_i$, which is the corner region with apex $\hh_j$ in the induced arrangement $\H'\mathrel{\mathop:}= (\H \setminus
\{H_i\}) \bigcap H_i$ of $d$ hyperplanes in $H_i \simeq \R^{d-1}$. Indeed, since $ \yy_j \in C_j = \bigcap_{k \in [d+1]\setminus \{j\}} H_k^+$
we have $\overline{\yy_i\yy_j} \subseteq \bigcap_{k \in [d+1]\setminus
\{i,j\}}H_k^+$, thus $\zz_j \in H_i \cap \big(
\bigcap_{k \in [d+1] \setminus \{i, j\}} H_k^+\big)$, which is by definition the corner region $C_j'$. We also observe that $\zz_i \in \Delta(\H')$. Therefore, by induction, $\zz_i$ is in the convex hull of the points
$\zz_j$ (for $j \in [d+1] \setminus \{i\}$). Since all these points $\zz_j$ are
by definition convex combinations of the points $\yy_1, \dots, \yy_{d+1}$ and
since $\hh_i$ is a convex combination of $\yy_i$ and $\zz_i$
we deduce that $\hh_i$ is in the simplex determined by $\yy_1, \dots, \yy_{d+1}$ as required.
\end{proof}


\section{Upper bound}
\label{s:ub}
The goal of this section is to give an exponentially decreasing upper bound on
$\csup_d$. As we sketched in the introduction, we set $X_i$, for $i \in [d+1]$,
to be a set of $n$ points uniformly distributed in the unit $d$-ball $B^d$.
We will explain later what we mean exactly by a uniform distribution.
The idea is that if $A$ is a `sufficiently nice' subset of $B^d$ then $\frac{\Vol(A)}{\Vol(B^d)}$ is
approximately equal to $\frac{|X_i \cap A|}{|X_i|}$.

By a \emph{generic Pach's configuration} we mean a collection $(Y_1,\dots,
Y_{d+1}, \pp)$ of $d+1$ finite pairwise disjoint nonempty sets $Y_i$ and a point $\pp$
not belonging to any $Y_i$ such that the set $Y_1 \cup \cdots \cup Y_{d+1} \cup
\{\pp\}$ is in general position and $\pp$ belongs to all $(Y_1, \dots,
Y_{d+1})$-simplices.

Note that if we consider $(Y_1, \dots, Y_{d+1}, \pp)$ as the output of
Theorem~\ref{t:pach} we need not obtain a generic Pach's configuration even if
$X\mathrel{\mathop:}= X_1 \cup \dots \cup X_{d+1}$ is in general position, since the point $\pp$ might
be on some of the hyperplanes determined by $X$. In such case, forgetting few
points only, we can still get a generic Pach's configuration; this is shown in Lemma~\ref{l:shrinkY}.

In Lemma~\ref{l:shrinkY} we require a stronger notion of general
position, which generalizes the following situation in the
plane. Let $X$ be a set in $\R^2$ and let $\ell_1 = \aa_1\bb_1$, $\ell_2 = \aa_2\bb_2$, and $\ell_3 =
\aa_3\bb_3$ be three lines in the plane determined by six distinct points of
$X$. Then we require that these three lines do not meet in a point.

In general, we say that a set $X$ of points in $\R^d$ satisfies \emph{condition (G)} if 

\begin{enumerate}
\item[1)] $X$ is in general position, and 
\item[2)] whenever $X_1, \dots, X_{d+1}$ are pairwise disjoint subsets of $X$, each of size at most $d$,
  then $\aff(X_1) \cap \cdots \cap \aff(X_{d+1})=\emptyset$. Here $\aff(X_i)$ denotes the affine hull of $X_i$. 
\end{enumerate}

For every set $X'$ that does not satisfy (G), we may obtain a set satisfying (G) by an arbitrarily small
perturbation of points in $X'$.

\begin{lemma}
  \label{l:shrinkY}
  Let $Y'_1, \dots, Y'_{d+1}$ be $d + 1$ finite pairwise disjoint sets of
  size at least $d + 1$ such
  that $Y'_1 \cup \cdots \cup Y'_{d+1}$ satisfies condition (G). Let $\pp'$ be a
  point contained in all $(Y'_1, \dots, Y'_{d+1})$-simplices. Then there are
  subsets $Y_i \subseteq Y'_i$ for $i \in [d+1]$ such that $|Y_i| \geq |Y'_i| -
  d$, and a point $\pp \in \R^d$ such that $(Y_1, \dots, Y_{d+1}, \pp)$ is a generic Pach's configuration.
\end{lemma}

\begin{remark}
  Condition~(G) is set up in such a way that the proof of Lemma~\ref{l:shrinkY}
  is simpler. Another approach would be to assume only the (standard and more
  intuitive) 
  general position instead of condition~(G). This would, however, yield a more complicated proof of
  Lemma~\ref{l:shrinkY} with a worse bound $|Y_i| \geq |Y'_i| - f(d)$, where
  we could achieve $f(d)$ to be slightly less than $2^d$. However, any function of
  $d$ would be fully sufficient for our needs.
\end{remark}

\begin{proof}
  Let $\Delta_1, \dots, \Delta_k$ be a maximal collection of $(Y'_1, \dots,
  Y'_{d+1})$-simplices such that $\pp'$ is on the boundary of each $\Delta_i$
  for $i \in [k]$ and any two simplices of this collection have disjoint
  vertex sets.

  Let $F_i$ be the set of vertices of a proper face of $\Delta_i$ containing $\pp'$. Since $\pp'\in\aff(F_1) \cap \cdots \cap \aff(F_k)$, condition $(G)$ implies that $k \le d$.

  Now, we remove all vertices of $\Delta_1, \dots, \Delta_k$ from each $Y'_i$, obtaining sets
  $Y_i$, removing at most $d$ points from each $Y'_i$. Then $\pp'$ is in the
  interior of all $(Y_1, \dots, Y_{d+1})$-simplices due to the maximality of
  $\Delta_1, \dots, \Delta_k$. By a small perturbation of $\pp'$ we get a point
  $\pp$ still in the interior of all $(Y_1, \dots, Y_{d+1})$-simplices. Then
  $(Y_1, \dots, Y_{d+1}, \pp)$ is the required generic Pach's configuration.
\end{proof}

The main idea of our proof of the upper bound is that if $(Y_1, \dots, Y_{d+1}, \pp)$ is a
generic Pach's configuration in the unit ball $B^d$, 
then some $Y_i$ is contained in a tiny part
of the ball. 
By $\beta_d$ we denote the volume of the unit ball $B^d$.

\begin{proposition}
\label{p:in_tiny_part}
  Let $(Y_1, \dots, Y_{d+1}, \pp)$ be a generic Pach's configuration such that
  $Y_1 \cup \cdots \cup Y_{d+1} \cup \{\pp\}$ is a subset of $B^d$.
Then
    there is an arrangement of hyperplanes $\H = (H_1, \dots, H_{d+1})$ in
    general position such that each $Y_i$ belongs to the corner region $C_i = C_i(\H)$
    (see the definitions in Section~\ref{s:corner}). The smallest of the
    volumes $\Vol(C_i \cap B^d)$ is at most $2^d \msa(\Delta(\H))
    \beta_d$ (we recall that $\msa$ denotes the minimum solid angle).
\end{proposition}

\ifobsolete
\gray{
\begin{oproposition}
\label{p:in_tiny_part_obsolete}
  Let $(Y_1, \dots, Y_{d+1}, \pp)$ be a generic Pach's configuration such that
  $Y_1 \cup \cdots \cup Y_{d+1} \cup \{\pp\}$ is a subset of $B^d$.
  Let $\alpha$ be the distance of $\pp$ and $\OO$. Then

  \begin{enumerate}
  \item[{\rm (i)}]
  there is $\ell \in [d + 1]$ such that $Y_\ell$ is contained in a
  solid cap cut out of $B^d$ of volume at most $(1 - \alpha^2)^{d/2} \beta_d$; and
  \item[{\rm (ii)}]
    there is an arrangement of hyperplanes $\H = (H_1, \dots, H_{d+1})$ in
    general position such that each $Y_i$ belongs to the corner region $C_i = C_i(\H)$
    (see the definitions in Section~\ref{s:corner}). The smallest of the
    volumes $\Vol(C_i \cap B^d)$ is at most $(1+\alpha)^d \msa(\Delta(\H))
    \beta_d$ (we recall that
    $\msa$ denotes the minimum solid angle).
  \end{enumerate}
\end{oproposition}
}

\gray{
For a proof we need the following two properties of generic Pach's configurations. 

\begin{olemma}
\label{l:cut_hyp_obsolete}
  Let $(Y_1, \dots, Y_{d+1}, \pp)$ be a generic Pach's configuration. Let $H$
  be any hyperplane passing through $\pp$. Then for any of the two open
  halfspaces determined by $H$ there is $\ell \in [d+1]$ such that $Y_\ell$ is fully
  contained in that halfspace.
\end{olemma}

\begin{proof}
  Let $H^+$ be the closed halfspace opposite to the open halfspace in which we
  look for $Y_\ell$. Suppose for contradiction that each $Y_i$ meets $H^+$. Let $\yy_i$ be a point from $Y_i \cap H^+$, for every $i \in [d+1]$.
  Since $\pp$ belongs to the simplex $\yy_1\yy_2\dots\yy_{d+1}$, it belongs to the convex hull of those $\yy_i$ that are in
  $H$. This contradicts the general position condition of a generic Pach's
  configuration.
\end{proof}
}
\fi

For a proof we need the following property of generic Pach's configurations. 

\begin{lemma}\label{lemma_generic_strict_separation}
  Let $(Y_1, \dots, Y_{d+1}, \pp)$ be a generic Pach's configuration. Then for every $i \in [d+1]$ there is a
  hyperplane $H_i$ strictly separating $\pp$ from
  $\widehat Y_i$. (We recall that $\widehat Y_i = \bigcup_{j \in [d+1] \setminus
  \{i\}} Y_j$.) Moreover, the hyperplanes $H_i$ can be chosen in such a way that the arrangement $\H
  = (H_1, \dots, H_{d+1})$ is in general position, $\pp \in \Delta(\H)$ and $Y_i \subseteq \interior (C_i(\H))$ for any $i \in
    [d+1]$.
\end{lemma}

\begin{proof}
  Suppose for contradiction that for some $i \in [d+1]$ the point $\pp$
  is not strictly separated from $\widehat Y_i$ by a hyperplane. That means that
  $\pp$ belongs to the convex hull $\conv(\widehat Y_i)$. Consequently, there are
  points $\zz_j \in \conv(Y_j)$ for $j \in [d+1] \setminus \{i\}$ such that
  $\pp$ is a convex combination of them. (Indeed, consider $\pp$ as a convex
  combination of points from $\widehat Y_i$ and put together points of each $Y_j$
  with appropriate weights.)

  Let $H$ be a hyperplane passing through the points $\zz_j$ for $j \in [d+1]
  \setminus \{i\}$.
  In particular, $\pp$ belongs to $H$.
  Let $\yy^+_i$ be a point of $Y_i$ and let
  $H^+$ and $H^-$ be the closed halfspaces determined by $H$ chosen in such a
  way that $\yy^+_i \in H^+$.
  For each $j \in [d+1] \setminus \{i\}$ we can find a point $\yy^+_j$ in $H^+
  \cap Y_j$ since $\conv(Y_j) \cap H \neq \emptyset$. Let $\Delta \subseteq
  H^+$ be the $(Y_1, \dots, Y_{d+1})$-simplex with vertices $\yy_j^+$
  for $j \in [d+1]$; see Figure~\ref{f:p_on_boundary}. Since $\Delta \subseteq H^+$ and
  since $\pp$ belongs to $H$, $\pp$ cannot be in the interior of
  $\Delta$. This contradicts our genericity assumption.

  It follows that there is a hyperplane $H_i$ strictly separating $\pp$ from
  $\widehat Y_i$. Finally, we rotate the hyperplanes $H_i$ a little bit, so that we keep their separation property and get an arrangement $\H$ in general position. Since $\pp$ is in all
  $(Y_1, \dots, Y_{d+1})$-simplices, Lemma~\ref{l:separate} implies that $\pp \in
  \Delta(\H)$ and that each $Y_i$ belongs to the interior of the corner region $C_i(\H)$. 
\end{proof}

\begin{proof}[Proof of Proposition~\ref{p:in_tiny_part}]
Let $\H = (H_1, \dots, H_{d+1})$ be the arrangement of hyperplanes from Lemma~\ref{lemma_generic_strict_separation}.
Since each $Y_i$ belongs to the corner region $C_i=C_i(\H)$, it remains to
bound the smallest of the volumes $\Vol(C_i \cap B^d)$. We refer to
Figure~\ref{f:vol_ci}, which illustrates the rest of the proof.
We use the same notation for the vertices of $\Delta(\H)$ as in Section~\ref{s:corner}.
  We fix $\ell \in [d+1]$ such that the solid angle $\vartheta$ at vertex $\hh_\ell$ is
  the minimum of all solid angles of $\Delta(\H)$. For each $i \in [d+1] \setminus \{\ell\}$, let $H'_i$ be a
  hyperplane parallel to $H_i$ passing through $\pp$ and let $C$ be the cell of the arrangement of hyperplanes $(H'_i)_{i
  \in [d+1] \setminus \{\ell\}}$ that contains $\hh_\ell$. Then $C$ contains
  $C_\ell$ and moreover
  $C \cap B(\pp, 2)$ contains $C_\ell \cap B^d$ since $B(\pp, 2)$ contains
  $B^d$. The
volume of $C \cap B(\pp, 2)$ is $2^d \vartheta = 2^d \msa(\Delta(\H)) \beta_d$. 
This gives the required upper bound.
\end{proof}

\ifobsolete
\gray{
\begin{proof}[Obsolete Proof of Proposition~\ref{p:in_tiny_part_obsolete}]
(i)
  Consider the hyperplane $H$ perpendicular to the line
  $\OO\pp$ and passing through $\pp$. By Lemma~\ref{l:cut_hyp_obsolete} there is $\ell$ such that $Y_{\ell}$ is in the
  solid cap cut by $H$. The volume of this solid cap is at most 
  $(1 - \alpha^2)^{d/2}\beta_d$ since it fits into a ball of radius $\sqrt{1 -
  \alpha^2}$; see Figure~\ref{f:cib}. 

\begin{figure}
\begin{center}
  \includegraphics{cap_in_ball.eps}
\end{center}
\caption{(Obsolete) Cap fits into a smaller ball}
\label{f:cib}
\end{figure}
\fi

\begin{figure}
\begin{center}
  \includegraphics{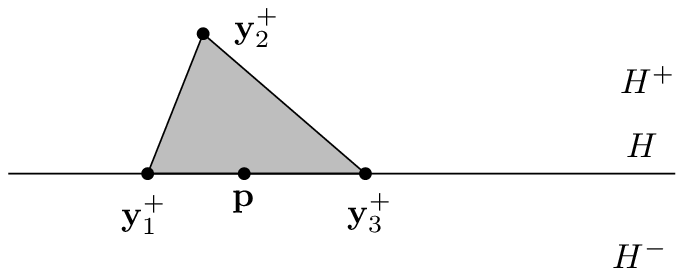}
\end{center}
\caption{The point $\pp$ cannot be in the interior of $\Delta$.}
\label{f:p_on_boundary}
\end{figure}

\ifobsolete
(ii) Let $\H = (H_1, \dots, H_{d+1})$ be the arrangement of hyperplanes from Lemma~\ref{lemma_generic_strict_separation}.
Since each $Y_i$ belongs to the corner region $C_i=C_i(\H)$, it remains to
bound the smallest of the volumes $\Vol(C_i \cap B^d)$. We refer to
Figure~\ref{f:vol_ci}, which illustrates the rest of the proof.
We use the same notation for the vertices of $\Delta(\H)$ as in Section~\ref{s:corner}.
  We fix $\ell \in [d+1]$ such that the solid angle $\vartheta$ at vertex $\hh_\ell$ is
  the minimum of all solid angles of $\Delta(\H)$. For each $i \in [d+1] \setminus \{\ell\}$, let $H'_i$ be a
  hyperplane parallel to $H_i$ passing through $\pp$ and let $C$ be the cell of the arrangement of hyperplanes $(H'_i)_{i
  \in [d+1] \setminus \{\ell\}}$ that contains $\hh_\ell$. Then $C$ contains
  $C_\ell$ and moreover, since the distance of the origin and $\pp$ is
  $\alpha$,
  we deduce that $C \cap B(\pp, 1 + \alpha)$ contains $C_\ell \cap B^d$. The
  volume of $C \cap B(\pp, 1 + \alpha)$ is $(1 + \alpha)^d \vartheta = (1
  + \alpha)^d \msa(\Delta(\H)) \beta_d$. This gives the required upper bound.
\end{proof}
}
\fi

\begin{figure}
\begin{center}
  \includegraphics{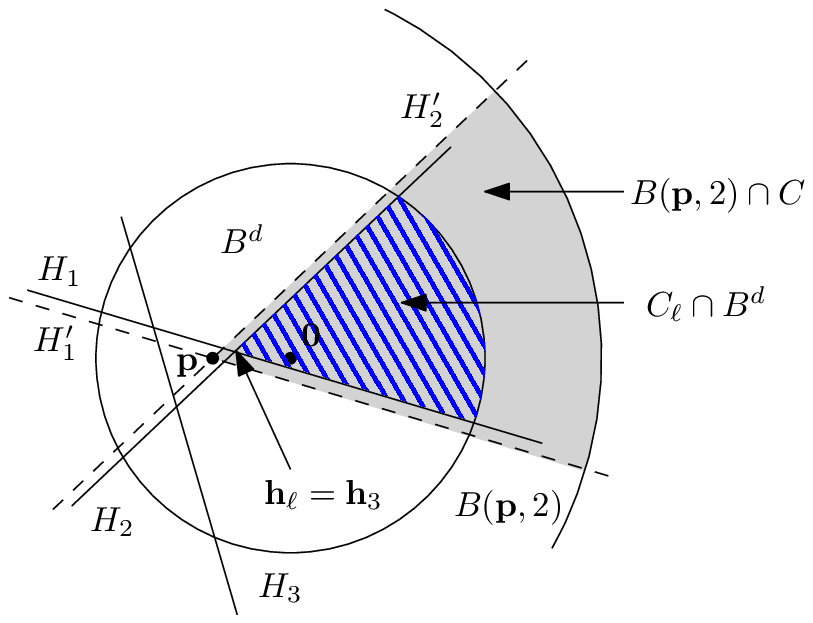}
\end{center}
\caption{$C \cap B(\pp, 1 + \alpha)$ contains $C_\ell \cap B^d$}
\label{f:vol_ci}
\end{figure}

Now we have all tools to prove the upper bound.

\begin{proof}[Proof of Theorem~\ref{t:bounds}(1)]
  Let $u(d)$ be the upper bound function on the minimum solid angle from Theorem~\ref{t:msa}, that is, 
  $\msa(\Delta) \leq u(d)$ for any simplex $\Delta$ and $u(d) \leq e^{-(1/2 -
  o(1))d \ln d}$. Let $g(d) \mathrel{\mathop:}= 2^d u(d)$. This value is still of order
  $e^{-(1/2 - o(1))(d \ln d)}$. In order to prove
  Theorem~\ref{t:bounds}(1), it is sufficient to show that Theorem~\ref{t:pach}
  cannot be valid with $c_d = g(d) + \zeta$ for any $\zeta > 0$.
  For contradiction, we assume that Theorem~\ref{t:pach} is valid
with such value of $c_d$.

Take a very small $\varepsilon > 0$ and tile $\R^d$ with hypercubes of side
$\varepsilon$. Let $\Q$ be the set of the hypercubes in the tiling
that intersect the interior of the unit ball $B^d$. For every $Q \in \Q$ and every $i \in [d+1]$ we select
exactly one point from $\interior (Q) \cap \interior (B^d)$ and add it into $X_i$, in such a way that the set $X_1 \cup \cdots \cup
X_{d+1}$ satisfies condition~(G). This finishes the construction of the
sets $X_i$.
  
Let $n\mathrel{\mathop:}=|\Q|$. By the construction, the size of each of the sets $X_i$ is $n$. Since $\bigcup\Q$ fits into a ball of radius $(1 + \varepsilon \sqrt d)$, we observe that $n$
  is well approximated in terms of the volume $\beta_d$ of
  $B^d$ as follows:
\begin{equation}
\label{e:bound_xi}
    \frac{1}{\varepsilon^d} \beta_d \leq n \leq \frac{1}{\varepsilon^d} \big(1
    + \varepsilon \sqrt d\big)^d \cdot\beta_d.
\end{equation}

We apply Theorem~\ref{t:pach} with $c_d = g(d) + \zeta$ to our sets $X_i$ and obtain sets $Y'_1, \dots, Y'_{d+1}$ and a point $\pp'$ as an output.
  If $\varepsilon$ is small enough, then $n$ is large enough so that
  Lemma~\ref{l:shrinkY} yields a generic Pach's configuration $(Y_1, \dots, Y_{d+1}, \pp)$ where
  $Y_i \subseteq X_i$ and $|Y_i| > (g(d) + \zeta / 2 )|X_i|$ for every $i \in [d+1]$.

  By Proposition~\ref{p:in_tiny_part} and Theorem~\ref{t:msa}, there is an
  $\ell \in [d+1]$ such that $Y_\ell$ is contained in the region
  $G\mathrel{\mathop:}= C_{\ell}\cap B^d$ with volume at most $2^d
  u(d)\beta_d = g(d) \beta_d$.

  We want to bound the number of points in $Y_\ell$ by the volume of 
  $G$. Let $\Q_\ell$ be a subset of $\Q$ consisting of those cubes that
  meet the interior of $G$. Note that
  \begin{equation}
    \label{e:YQl}
    |Y_\ell| \leq |\Q_\ell|.
  \end{equation}
  We further split $\Q_\ell$ into two disjoint sets $\Q_\ell^\partial$ and $\Q_\ell^{\interior}$ where $\Q_\ell^\partial$
  contains those cubes that meet the boundary of $G$ and
  $\Q_\ell^{\interior}$ contains those cubes that are fully contained in the
  interior of $G$. See Figure~\ref{f:Q_cubed}.

\begin{figure}
\begin{center}
  \includegraphics{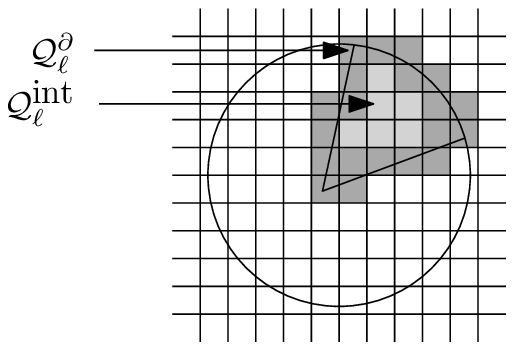}
\end{center}
\caption{Splitting $\Q_\ell$ into $\Q_\ell^\partial$ and $\Q_\ell^{\interior}$. With decreasing $\varepsilon$, the volume of 
  $\bigcup \Q_\ell^\partial$ tends to $0$}
\label{f:Q_cubed}
\end{figure}

 We have an obvious upper bound on the size of $\Q_\ell^{\interior}$:
 \begin{equation}
\label{e:bound_qint}
|\Q_\ell^{\interior}| \leq \frac1{\varepsilon^d} \Vol(G) \leq
\varepsilon^{-d}g(d)\beta_d.
 \end{equation}

For the size of $\Q_\ell^{\partial}$ we can get the following upper bound. Each cube
of $\Q_\ell^{\partial}$ belongs to the $(\varepsilon \sqrt d)$-neighborhood
$N_\varepsilon$
of the boundary $\partial G$ of $G$. The $(d-1)$-dimensional volume
of $\partial G$ can be bounded by some function
$f(d)$ depending only on $d$ (note that $G$ was obtained by cutting $B^d$ at
most $d$-times). Therefore
 \begin{equation}
\label{e:bound_lim}
   \lim_{\varepsilon\to 0} \Vol(N_\varepsilon) = 0
 \end{equation}
considering $d$ fixed. In addition,
 \begin{equation}
\label{e:bound_qboundary}
|\Q_\ell^{\partial}| \leq \frac1{\varepsilon^d} \Vol(N_\varepsilon).
 \end{equation}

 Combining $|X_\ell| = n$ with \eqref{e:bound_xi}, \eqref{e:YQl}, \eqref{e:bound_qint} and
\eqref{e:bound_qboundary} yields
\[
\frac{|Y_\ell|}{|X_\ell|} \leq \frac{|\Q_\ell|}n \leq
\frac{\varepsilon^{-d}g(d)\beta_d +
\varepsilon^{-d} \Vol(N_\varepsilon)}{\varepsilon^{-d}\beta_d} = 
g(d) + \frac{\Vol(N_\varepsilon)}{\beta_d}.
\]
Using~\eqref{e:bound_lim}, this is a contradiction with
$\frac{|Y_\ell|}{|X_\ell|} > g(d) +
\frac{\zeta}2$ if $\varepsilon$ is small enough.

\end{proof}

\section{Lower bound}
\label{s:lb}
In this section we prove Theorem~\ref{t:bounds}(2). We reuse many steps form
Pach's original proof~\cite{pach98} and we also follow an exposition of Pach's
proof by Matou\v{s}ek~\cite[Chapter~9]{matousek02}.

\begin{lemma}[Few separations]
\label{l:weak-sep}
  Let $S_1,\ldots, S_{d+1}$ be disjoint finite sets of points in $\R^d$ and
  let $\pp$ be a point in $\R^d$ such that $S_1\cup S_2\cup\ldots\cup S_{d+1} \cup
  \{\pp\}$ is in general position. Then there exist sets $Y_1\subseteq S_1,\ldots, Y_{d+1}\subseteq S_{d+1}$ satisfying
    \begin{enumerate}
	    \item $|Y_i|\geq \frac{1}{2^d} |S_i|$, and 
	    \item the point $\pp$ either lies in all $(Y_1,\ldots, Y_{d+1})$-simplices, or in none of them.
    \end{enumerate}
\end{lemma}

\begin{proof}
  We will reduce the sizes of the sets $S_i$ in $d+1$ steps; after these steps
  we obtain the required sets $Y_i$.
  For each $i \in [d+1]$ we set $S_i^{(0)} \mathrel{\mathop:}= S_i$. In the $j$th step
  we construct a hyperplane $H'_j$ and sets $S_i^{(j)}$ for all $i \in [d+1]$ 
  with the following properties.
  \begin{enumerate}[(i)]
  \item $S_i^{(j)} \subseteq S_i^{(j-1)}$ for $i,j \in [d+1]$;
  \item $|S_i^{(j)}| \geq |S_i^{(j-1)}|/2$ for $i,j \in [d+1]$, $i \neq j$;
  \item $S_j^{(j)} = S_j^{(j-1)}$ for $j \in [d+1]$; and
  \item $H'_j$ strictly separates $\pp$ from $S_i^{(j)}$ for $i,j \in [d+1]$,
    $i \neq j$.
\end{enumerate}
 This can be easily done inductively using the ham sandwich theorem.
 In the $j$th step we assume that we have already constructed the sets
 $S_i^{(j')}$ and the hyperplanes $H'_{j'}$ for $j' < j$.
By the general position variant of the ham sandwich
theorem~\cite[Corollary 3.1.3]{Matousek03} there is a hyperplane $H''_j$
simultaneously bisecting the $d$ sets $S^{(j-1)}_i$ for $i \neq j$.
That is, both open halfspaces determined by $H''_j$ contain at least $\lfloor
|S_i^{(j-1)}| / 2\rfloor$ points of each $S_i^{(j-1)}$ for $i \in [d+1]
\setminus \{j\}$. 
To obtain the required conclusion, we would like to choose $S_i^{(j)}$ to be the half of $S_i^{(j-1)}$ that belongs
to the opposite halfspace than $\pp$. 

We just have to be careful when $\pp$ actually belongs to $H''_j$ or
when $H''_j$ intersects some $S_i^{(j-1)}$ for $i \in [d+1]
\setminus \{j\}$. If $\pp \in H''_j$, we consider
the (possibly empty) set $U\mathrel{\mathop:}=H''_j\cap (S_1 \cup \cdots \cup S_{d+1})$. We realize that the flat determined by $U$ (that is, the affine hull of
$U$) is strictly contained in $H''_j$ and $\pp$ does not belong to this flat,
both by the general position assumption on $\{\pp\} \cup U$.
Therefore, we can perturb $H''_j$ a bit so that it still contains $U$ but it
avoids $\pp$ and no other point of $S_1 \cup \cdots \cup S_{d+1}$ switched the
side. So we can assume that $\pp$ does not belong to $H''_j$.

As $\pp$ does not belong to $H''_j$, we consider the
hyperplane $H'_j$ obtained by shifting $H''_j$ by a small bit towards $\pp$. For
$i \in [d+1]\setminus j$ we set $S^{(j)}_i$ to be the subset of $S^{(j-1)}$
belonging to the open halfspace on the other side of $H'_j$ than $\pp$. We also
set $S^{(j)}_j \mathrel{\mathop:}= S^{(j-1)}_j$. Then these sets satisfy the required
conditions (i)--(iv).

Finally, we set $Y_i \mathrel{\mathop:}= S^{(d+1)}_i$ for $i \in [d+1]$. Then $Y_i \subseteq
S_i$ and $|Y_i| \geq \frac1{2^d} |S_i|$ by (i), (ii) and (iii). We slightly
perturb the hyperplanes $H'_j$ obtaining new hyperplanes $H_j$ in general
position such that each $H_j$ still strictly separates $\pp$ and $Y_i$. Letting
$\H$ be the arrangement of these hyperplanes we get either $\pp \in
\Delta(\H)$ or not.

In the first case Lemma~\ref{l:separate} and Lemma~\ref{l:in_simplex} imply
that $\pp$ is in all $(Y_1, \dots, Y_{d+1})$-simplices. In the second case
Lemma~\ref{l:separate} implies that $\pp$ is in no $(Y_1, \dots,
Y_{d+1})$-simplex.
\end{proof}

The last tool we need for the proof of Theorem~\ref{t:pach} is the weak
hypergraph regularity lemma. We will be given a $k$-partite $k$-uniform
hypergraph $\mathbf{H}$ on the vertex set $X_1
 \cup \cdots \cup X_k$, where the sets $X_i$ are pairwise disjoint and each edge of the hypergraph contains exactly
 one point from each of the sets $X_i$.
 For any subsets $Y_i \subseteq X_i$, $i \in [k]$, we define $e(Y_1, \dots,
 Y_k)$ as the number of edges in the subhypergraph $\mathbf{H}[Y_1,\dots,Y_k]$ induced by $Y_1, \dots, Y_k$.
 We also define the \emph{density} function
 \[
 \rho(Y_1, \dots, Y_k) \mathrel{\mathop:}= \frac{e(Y_1,\dots,Y_k)}{|Y_1|\cdots|Y_k|}
 \]
 as the ratio of the number of edges in $\mathbf{H}[Y_1,\dots,Y_k]$ and the number
 of all possible edges in a $k$-partite hypergraph with vertex set $Y_1 \cup
 \cdots \cup Y_k$. We also set $\rho(\mathbf{H}) \mathrel{\mathop:}= \rho(X_1, \dots, X_k)$.

 \begin{theorem}[Weak regularity lemma for hypergraphs~\cite{pach98}; see
   also~{\cite[Theorem 9.4.1]{matousek02}}]\label{t:weak_reg_lemma}
 Let $\mathbf{H}$ be a $k$-partite $k$-uniform hypergraph on a vertex set $X_1
 \cup \cdots \cup X_k$, where $|X_i| = n$ for $i \in [k]$. Suppose that its
 edge density satisfies $\rho(\mathbf H) \geq \beta$ for some $\beta > 0$. Let $0
 < \varepsilon < \frac12$. Suppose also that $n$ is sufficiently large in terms
 of $k$, $\varepsilon$ and $\beta$.

 Then there exist subsets $S_i \subseteq X_i$ of equal size $|S_i|=s\geq
 \beta^{1/\varepsilon^k}n$, for any $i \in [k]$ such that
 \begin{enumerate}
    \item (High density) $\rho(S_1,\ldots,S_k)\geq \beta$, and
    \item (Edges on large subsets) $e(Y_1,\ldots,Y_k)>0$ for any $Y_i \subseteq S_i$ with $|Y_i|\geq \varepsilon s, i=1,2,\ldots,k$.
 \end{enumerate}
\end{theorem}

We are finally ready to prove the lower bound on the maximum Pach's constant from Theorem~\ref{t:pach}.

\begin{proof}[Proof of Theorem~\ref{t:bounds}(2)]

It is convenient to start the proof with additional assumptions. Later on we
will show how to remove these assumptions. We start assuming that $X_1 \cup
\cdots \cup X_{d+1}$ is in general position and also assuming that the size $n$ of
the sets $X_i$ is large enough, that is, $n \geq n_0$, where $n_0$ depends only on
$d$.

By Theorem~\ref{t:karasev}, there is a point $\pp$ contained in the interior of
at least $\frac{1}{(d+1)!} n^{d+1} - O(n^d)$  $(X_1,\ldots,X_{d+1})$-simplices. We
perturb the point $\pp$ a little so that $X_1 \cup \cdots X_{d+1} \cup \{\pp\}$
is in general position but $\pp$ does not leave the interior of any
$(X_1,\ldots,X_{d+1})$-simplex during the perturbation. We require that $n_0$
is large enough so that $\pp$ actually belongs to the interior of at least
$\frac{1}{2^{d^2}} n^{d+1}$ $(X_1,\ldots,X_{d+1})$-simplices, using a very rough
estimate $(d+1)! < 2^{d^2}$ (a better estimate would not improve the bound
significantly).

 Next, we consider the $(d+1)$-partite hypergraph $\mathbf H$ with vertex set $X_1
 \cup X_2 \cup \cdots \cup X_{d+1}$ whose edges are precisely the
 $(X_1,\ldots,X_{d+1})$-simplices containing the point $\pp$.
 Let $\varepsilon = \frac{1}{2^d}$ and let us further require that
 $n_0$ is large enough so that the assumptions of
 Theorem~\ref{t:weak_reg_lemma} are met.
 We apply the weak regularity lemma (Theorem~\ref{t:weak_reg_lemma}) to
 $\mathbf H$. Note that $\beta \geq \frac{1}{2^{d^2}}$. This yields sets $S_i \subseteq
 X_i$ with size $|S_i|=s\geq \beta^{1/\varepsilon^{d+1}}n$, and such that any
 subsets $Y_i \subseteq S_i$ of size at least $\varepsilon s$ induce an
 edge; that is, there is a $(Y_1,\ldots, Y_{d+1})$-simplex containing
 the point $\pp$.

 Finally, we apply Lemma \ref{l:weak-sep} with the sets $S_1,\ldots,S_{d+1}$
 and point $\pp$. We obtain sets $Y_i \subseteq S_i$ with $|Y_i|\geq
 \frac{1}{2^d}s = \varepsilon s$.
 Moreover, the point $\pp$ either lies in all $(Y_1,\ldots,
 Y_{d+1})$-simplices, or in none of them. But the latter possibility is
 excluded by the fact that $Y_i$ are large enough.

 Because $\csup_1=1/2$, we assume $d\geq 2$ in the following calculations.
 So we obtained the desired sets $Y_i$'s of size $c_d|X_i|$, where
 \[
 c_d \geq \frac{1}{2^d}\beta^{1/\varepsilon^{d+1}}
 \geq
 \frac{1}{2^d}\cdot \left(\frac 1{2^{d^2}}\right)^{2^{d(d+1)}} =
 2^{-d - d^2\cdot2^{d(d+1)} } \geq 2^{- 2^{d^2  + 3d}}.
 \]
 This finishes the proof under the assumptions that $X_1 \cup \cdots \cup
 X_{d+1}$ is in general position and $n \geq n_0$.

 First, by a standard compactness argument we can remove the
 general position assumption. Here we can even assume that $X_i$ are multisets,
 that is, some of the points can be repeated more than once. Indeed, we choose sets
 $X^{(m)}_i$ of size $n$ such that $X_1^{(m)} \cup \cdots \cup X_{d+1}^{(m)}$ is in general
 position for every positive integer $m$ and such that $X^{(m)}_i$ converges to
 $X_i$. We obtain the corresponding sets $Y_i^{(m)}$ and Pach points $\pp^{(m)}$
 using the general position version of the theorem. Since the sizes of the sets $X_i^{(m)}$ are uniformly bounded by $n$,
 there is an infinite increasing sequence $(m_k)$ such that for every $i\in [n+1]$, the sequence $Y_i^{(m_k)}$
 converges to a certain set $Y_i \subseteq X_i$. Since all the sets $X^{(m)}_i$ belong to
 a compact region in $\R^d$, the sequence of Pach points $\pp^{(m_k)}$ has an accumulation point $\pp$. It is routine to check that the sets $Y_i$ and the point $\pp$ satisfy the required conditions.

 Next, we can remove the assumption $n \geq n_0$ in the following way. If $n <
 n_0$ we find an integer $m$ such that $m \cdot n \geq n_0$. We make multisets
 $X'_i$ where each $X'_i$ consists of points of $X_i$, each repeated $m$ times. Using the theorem for the sets $X'_i$, we find a point $\pp'$ and sets $Y'_i$ of sizes at least $c_d \cdot m \cdot n$.
 Forgetting the $m$-fold repetitions in $Y'_i$ we the get the required sets $Y_i$ of sizes
 at least $|Y'_i|/m$, and we set $\pp \mathrel{\mathop:}= \pp'$.
\end{proof}

\begin{remark}
  The argument at the end of the previous proof also shows that the assumption
  that all $X_i$ have equal size can be easily removed. Indeed, let $X_1, \dots,
  X_{d+1}$ be subsets
  of $\R^d$ of various sizes. We set $\gamma \mathrel{\mathop:}= |X_1|\cdots|X_{d+1}|$. We
  create multisets $X'_i$ where each point of $X_i$ is repeated $\gamma/|X_i|$
  times. That is, each $X'_i$ has size $\gamma$ and so we can find $\pp'$ and sets $Y'_i$ of sizes
  at least $c_d \gamma$. Forgetting the repetitions in
  $Y'_i$ we get sets $Y_i$ of sizes at least $c_d |X_i|$.
\end{remark}

\section{Measure version of Pach's theorem}\label{section_measure}

\subsection{Borel probability measures}

First we review some essential measure-theoretic background.
A sequence $\mu_n$ of Borel probability measures on $\R^d$ is \emph{weakly convergent} to a Borel probability measure $\mu$ on $\R^d$ if for every bounded continuous function $f:\R^d \rightarrow \R$, we have 
\[
\lim_{n\rightarrow \infty} \int_{\R^d} \!f\, \mathrm{d}\mu_n = \int_{\R^d} \!f\, \mathrm{d}\mu.
\] 
Alexandroff~\cite{Alexandroff43} established several equivalent definitions of weak convergence. The following one shows that it is sufficient to test the measure of closed sets.

\begin{theorem}[\cite{Alexandroff43}; see also{~\cite[Corollary 8.2.10]{bogachev07_II}}]\label{theorem_weak_closed}
A sequence $\mu_n$ of Borel probability measures on $\R^d$ is weakly convergent to a Borel probability measure $\mu$ on $\R^d$ if and only if for every closed set $F\subseteq \R^d$, we have 
\[
\limsup_{n\rightarrow \infty} \mu_n(F) \le \mu(F).
\]
\end{theorem}

The weak convergence of Borel probability measures on $\R^d$ can be also
defined as the convergence in the \emph{weak topology} on the space of Borel
probability measures on $\R^d$; see~\cite[Definition 8.1.2]{bogachev07_II}.
Moreover, this space is metrizable.

\begin{theorem}[{\cite[Theorem 8.3.2]{bogachev07_II}}]\label{theorem_weak_metric}
The weak topology on the space of Borel probability measures on $\R^d$ is generated by the L\'evy--Prohorov metric:
\[
d_P(\mu,\nu)\mathrel{\mathop:}=\inf\big\{\varepsilon>0; \text{ for every Borel set }B\subseteq
\R^d,  \nu(B)\le \mu(B^{\varepsilon})+\varepsilon \text{ and } \mu(B)\le
\nu(B^{\varepsilon})+\varepsilon\big\} 
\] 
where $B^{\varepsilon}=\{x\in\R^d;
\dist(x,B)<\varepsilon\}$.  
\end{theorem}

A measure $\mu$ on $\R^d$ is \emph{outer regular} if for every $\mu$-measurable set $S$ we have $\mu(S) = \inf\{\mu(U); S\subseteq U, U \text{ open} \}$.

\begin{lemma}[see~{\cite[Theorem 1.10.10 and Exercise 1.10.12]{Tao10_epsilon_I}}]\label{lemma_outer_regular}
Every Borel probability measure on $\R^d$ is outer regular.
\end{lemma}

The \emph{Dirac's measure} $\delta_x$ at $x\in \R^d$ is a measure on $\R^d$ satisfying $\delta_x(\{x\})=1$ and $\delta_x(\R^d\setminus\{x\})=0$.
It is well known that Borel probability measures can be approximated by finite linear combinations of Dirac's measures in the following sense.

\begin{lemma}\label{lemma_approx}
For every Borel probability measure $\mu$ on $\R^d$ there is a sequence of measures $\mu_n$ weakly convergent to $\mu$ such that each $\mu_n$ has the following form: $\mu_n=\sum_{i=1}^{k_n} c_{n,i}\delta_{x_{n,i}}$ where $c_{n,i} \in (0,1]$ and $x_{n,i} \in \R^d$.
\end{lemma}

\begin{proof}
By~\cite[Example 8.1.6 (i)]{bogachev07_II}, finite nonnegative convex combinations of Dirac's measures are dense in the space of Borel probability measures with the weak topology. Since this topological space is metrizable by Theorem~\ref{theorem_weak_metric}, every point $\mu$ has a countable base of open neighborhoods and the lemma follows.
\end{proof}

\begin{corollary}\label{cor_rational_approx}
For every Borel probability measure $\mu$ on $\R^d$ there is a sequence of probability measures $\mu'_n$ weakly convergent to $\mu$ such that each $\mu'_n$ is a finite nonnegative rational combination of Dirac's measures on $\R^d$.
\end{corollary}

\begin{proof}
For every $n$, let $\mu_n=\sum_{i=1}^{k_n} c_{n,i}\delta_{x_{n,i}}$ be the
measure from Lemma~\ref{lemma_approx}. For every $i\in[k_n]$, select a rational
number $c'_{n,i} \in ((1-1/n)\cdot c_{n,i},c_{n,i}]$. Let $c'\mathrel{\mathop:}=\sum_{i=1}^{k_n}
c'_{n,i}$. It is easy to see that $\mu'_n\mathrel{\mathop:}=\sum_{i=1}^{k_n}
(c'_{n,i}/c')\cdot\delta_{x_{n,i}}$ is a probability measure and that the
sequence $\mu'_n$ weakly converges to $\mu$, since for every bounded continuous
function $f:\R^d \rightarrow \R$ we have $(1-1/n) \int_{\R^d} \!f\,
\mathrm{d}\mu_n \le \int_{\R^d} \!f\, \mathrm{d}\mu'_n \le (n/(n-1))
\int_{\R^d} \!f\, \mathrm{d}\mu_n$.  
\end{proof}

Let $\mu$ be a finite nonnegative linear combination of Dirac's measures on $\R^d$. The \emph{support} $\supp(\mu)$ of $\mu$ is the set of points $x$ such that $\mu(\{x\})>0$. For our application it is convenient to approximate a given Borel measure with uniform discrete measures whose support is in general position.

\begin{corollary}\label{cor_approx_gp}
Let $\mu_1, \dots, \mu_{d+1}$ be Borel probability measures on $\R^d$. For every $i\in[d+1]$, there is a sequence of probability measures $\mu''_{i,n}$ weakly convergent to $\mu_i$ such that each $\mu''_{i,n}$ is of the form $(1/k_{i,n})\sum_{j=1}^{k_{i,n}} \delta_{x_{i,n,j}}$ where $x_{i,n,j} \in \R^d$,
and moreover, the supports $\supp(\mu''_{1,n}),\allowbreak \dots,\allowbreak \supp(\mu''_{d+1,n})$ are pairwise disjoint and $\supp(\mu''_{1,n}) \cup \dots \cup \supp(\mu''_{d+1,n})$ is in general position.
\end{corollary}

\begin{proof}
For every $i\in[d+1]$, let $\mu'_{i,n}$ be a sequence of measures from Corollary~\ref{cor_rational_approx} weakly convergent to $\mu_i$. Suppose that $\mu'_{i,n} = \sum_{j=1}^{k'_{i,n}} c'_{i,n,j}\delta_{x'_{i,n,j}}$. Since the coefficients $c'_{i,n,j}$ are rational, we have $c'_{i,n,j}=r_{i,n,j}/s_n$ for some positive integers $r_{i,n,j}$ and $s_n$.

For every $n$, we define the measures $\mu''_{i,n}$ as follows. For every $i\in[d+1]$ and for every $x'_{i,n,j}\in \supp(\mu'_{i,n})$, we select a set $X''_{i,n,j}$ of $r_{i,n,j}$ unique points, each of them  
at distance smaller than $1/n$ from $x'_{i,n,j}$, and such that for every fixed $n$, the set $\bigcup_{i,j} X''_{i,n,j}$ of all these $(d+1)\cdot s_n$ new points is in general position. For every $n$ and $i$, let $X''_{i,n}\mathrel{\mathop:}=\bigcup_{j} X''_{i,n,j}$ be the set of the $s_n$ new points created from the points $x'_{i,n,j}$. We set $\mu''_{i,n} \mathrel{\mathop:}= \sum_{x''\in X''_{i,n}} (1/s_n)\delta_{x''}$.

We use Theorem~\ref{theorem_weak_metric} to verify the convergence of the measures $\mu''_{i,n}$. We claim that $\mu'_{i,n}$ and $\mu''_{i,n}$ are $(1/n)$-close in the L\'evy--Prohorov metric. Indeed, for every Borel set $B\subseteq \R^d$ and for every point $x'_{i,n,j}\in \supp(\mu'_{i,n})$, if $x'_{i,n,j}\in B$ then $X''_{i,n,j}\subset B^{1/n}$. This implies that $\mu'_{i,n}(B) \le \mu''_{i,n}(B^{1/n})$. The inequality $\mu''_{i,n}(B) \le \mu'_{i,n}(B^{1/n})$ follows analogously.

Since $d_P(\mu'_{i,n},\mu''_{i,n})<1/n$ and $d_P(\mu'_{i,n}, \mu_i) \rightarrow 0$, we conclude that $d_P(\mu''_{i,n}, \mu_i) \rightarrow 0$ and the statement follows.
\end{proof}

\subsection{Proof of Theorem~\ref{t:pach_continuous}}

Let $\gamma(d)\mathrel{\mathop:}=2^{-2^{d^2 + 3d}}$. Let $\mu_{1}, \dots, \mu_{d+1}$ be Borel probability measures on $\R^d$. For $i\in[d+1]$, let $\mu''_{i,n}$ be the sequence of measures from Corollary~\ref{cor_approx_gp}.

For every $n$, we apply Theorem~\ref{t:bounds}(2) to the supports of the measures $\mu''_{1,n}, \dots, \mu''_{d+1,n}$.
We obtain sets $Y_{1,n}, \dots, Y_{d+1,n}$ and a point $\pp_n$ such that $Y_{i,n} \subseteq \supp(\mu''_{i,n})$, the point $\pp_n$ is in all $(Y_{1,n}, \dots, Y_{d+1,n})$-simplices and $\mu''_{i,n}(Y_{i,n})\ge \gamma(d)$. 
Moreover, we know from the proof of Theorem~\ref{t:bounds}(2) that there is an
arrangement $\H_n$ of $d+1$ hyperplanes in general position such that the sets
$Y_{i,n}$ are in the interiors of the corner regions of $\H_n$ and
$\pp_n$ is in the interior of the simplex $\Delta(\H_n)$ determined by this
arrangement. In this section, we denote the corner regions of $\H_n$ by
$Z_{i,n}$.

The key observation is that we can encode the output of
Theorem~\ref{t:bounds}(2) as a $(d+2)$-tuple of points of $\R^d$ that consists
of the vertices $\hh_{1,n}, \dots, \hh_{d+1,n}$ of the simplex $\Delta(\H_n)$
and the point $\pp_n$. In order to handle passing to the limit, we enrich this
data by a $((d+1)\cdot(2^d-1))$-tuple of vectors defined as follows.
For every $F\subseteq [d+1]$, let
$\sigma_{F,n}$ be the face $\conv(\{\hh_{j,n}; j\in F\})$ of $\Delta(\H_n)$.
In addition, if $i \in F$ and $F \neq \{i\}$,
let $\uu_{i,F,n}$ be a unit vector satisfying the following two conditions
(see Figure~\ref{f:uF}):

\begin{enumerate} 
\item[a)] The ray $\{\hh_{i,n}-\lambda\uu_{i,F,n}; \; \lambda\ge 0\}$ intersects the relative interior of $\sigma_{F,n}$.
\item[b)] Let $H_{i,F,n}$ be the affine hyperplane in the affine hull of $\sigma_{F,n}$ orthogonal to $\uu_{i,F,n}$ and containing $\hh_{i,n}$. Then $H_{i,F,n}\cap \sigma_{F,n} = \{\hh_{i,n}\}$. Equivalently, for every $j\in F\setminus \{i\}$, we have $\uu_{i,F,n} \cdot (\hh_{j,n}-\hh_{i,n}) <0$. Here we write $\uu\cdot\vv$ for the dot product of $\uu$ and $\vv$.
\end{enumerate}
In particular, if $F=\{i,j\}$, then $\uu_{i,F,n}=\vv_{i,j,n}\mathrel{\mathop:}=(\hh_{i,n}-\hh_{j,n})/\lVert\hh_{i,n}-\hh_{j,n}\rVert$. Here $\lVert \vv \rVert$ denotes the Euclidean norm of $\vv$.

The existence of the vector $\uu_{i,F,n}$ satisfying both conditions a) and b)
is not immediately obvious, especially when the dimension of the face
$\sigma_{F,n}$ is large. Let $C_{i,F,n} \mathrel{\mathop:}= (Z_{i,n}\cap
\aff(\sigma_{F,n}))-\hh_i$; that is, $C_{i,F,n}$ is the $(|F|-1)$-dimensional
convex cone with apex at the origin generated by the vectors $\vv_{i,j,n}$,
$j\in F\setminus \{i\}$. Condition a) now says that $\uu_{i,F,n}$ belongs to
the relative interior of $C_{i,F,n}$ if $|F| \geq 3$. Similarly, condition b)
says that $\uu_{i,F,n}$ belongs to the relative interior of the dual cone $C'_{i,F,n} \mathrel{\mathop:}= \{\yy \in \aff(\sigma_{F,n}); \; \yy \cdot \vv_{i,j,n} \ge 0, \; j\in F\setminus \{i\}\}$. The existence of $\uu_{i,F,n}$ thus follows from the following generalization of Farkas' lemma. 

\begin{lemma}
Let $C$ be a simplicial cone in $\mathbb{R}^k$ with apex in the origin,
where a \emph{simplicial cone} in $\mathbb{R}^k$ is a convex
hull of $k$ extremal rays emanating from the apex with linearly independent directions.
Let $C'\mathrel{\mathop:}= \{\yy \in \mathbb{R}^k; \; \yy \cdot \xx \ge 0 \text{ for all } \xx\in C\}$ be the dual cone of $C$. Then their intersection $C\cap C'$ has nonempty interior.
\end{lemma}

\begin{proof}
Suppose that the interior of $C\cap C'$ is empty. Since $C$ is simplicial both $C$ and $C'$ have dimension $k$ and thus nonempty interior.
By the non-strict version of the hyperplane separation theorem, there is a hyperplane $H$ separating the interiors of $C$ and $C'$ and passing through the origin. That is, there is a vector $\aa \in \mathbb{R}^k$ such that $\lVert \aa \rVert = 1$, $\aa\cdot\xx\ge 0$ for all $\xx \in C$, and $\aa\cdot\yy\le 0$ for all $\yy \in C'$. This implies that $\aa\in C'$, and consequently $\aa\cdot\aa \le 0$, which is a contradiction.
\end{proof}

Clearly, for every $F\subseteq [d+1]$ such that $i\in F$ and $F \neq \{i\}$, 
and for every
$\lambda\ge 0$, the point $\hh_{i,n}+\lambda\uu_{i,F,n}$ is contained in
$Z_{i,n}$. Moreover, $Z_{i,n}$ is the convex hull of the $d$ rays
$\{\hh_{i,n}+\lambda\vv_{i,j,n}; \; \lambda\ge 0\}$, for $j\in[d+1] \setminus \{i\}$. 

\begin{figure}
\includegraphics{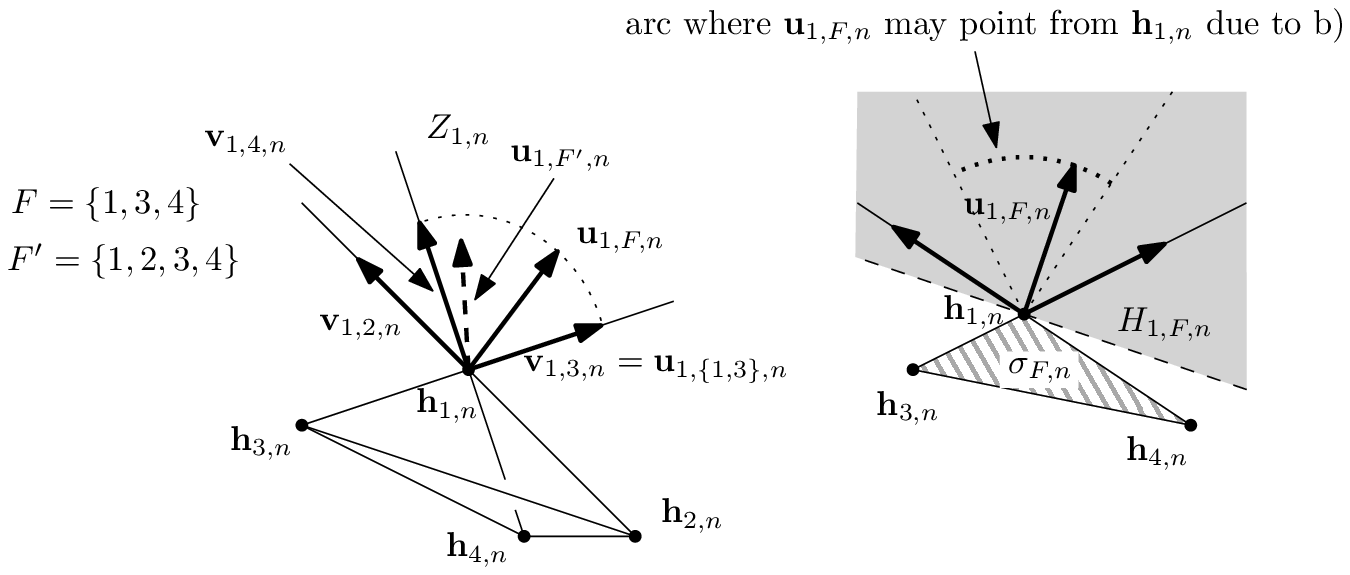}
\caption{Some of the (unit) vectors $\uu_{1,F,n}$ and $\vv_{1,j,n}$ in the
$3$-dimensional case. The right part of the picture shows the affine hull of $\sigma_{F,n}$ for $F = \{1,3,4\}$}
\label{f:uF}
\end{figure}

Let $\xx_n$ be the point in $\R^{(d+1)2^d+1}$ representing the ordered sequence of the points $\pp_n,\allowbreak \hh_{1,n},\dots,\allowbreak \hh_{d+1,n}$ and the vectors $\uu_{i,F,n}$.

Let $B$ be a closed ball centered in the origin such that for every $i\in[d+1]$, we have $\mu_{i}(B)>1-\gamma(d)/2$.
By Theorem~\ref{theorem_weak_closed}, there is an $n_0$ such that for every $n>n_0$ and for every $i\in[d+1]$, we have $\mu''_{i,n}(B)>1-\gamma(d)$.

We claim that for every $n>n_0$, the whole simplex $\Delta(\H_n)$ is contained
in $B$. Suppose the contrary. Then there is a point $\mathbf{q} \in \interior(
\Delta(\H_n))\setminus B$, which can be strictly separated from $B$ by a
hyperplane $H$. By Claim~\ref{c:cut_hyp} below applied to the generic Pach's
configuration $(Y_{1,n},\dots,Y_{d+1,n}, \mathbf{q})$, some of the sets
$Y_{i,n}$ is separated by $H$ from $B$. This is a contradiction as
$\mu''_{i,n}(B) + \mu''_{i,n}(Y_{i,n}) >1$.

\begin{claim}
\label{c:cut_hyp}
  Let $(Y_1, \dots, Y_{d+1}, \pp)$ be a generic Pach's configuration. Let $H$
  be any hyperplane passing through $\pp$. Then for any of the two open
  halfspaces determined by $H$ there is $\ell \in [d+1]$ such that $Y_\ell$ is fully
  contained in that halfspace.
\end{claim}

\begin{proof}
  Let $H^+$ be the closed halfspace opposite to the open halfspace in which we
  look for $Y_\ell$. Suppose for contradiction that each $Y_i$ meets $H^+$. Let $\yy_i$ be a point from $Y_i \cap H^+$, for every $i \in [d+1]$.
  Since $\pp$ belongs to the simplex $\yy_1\yy_2\dots\yy_{d+1}$, it belongs to the convex hull of those $\yy_i$ that are in
  $H$. This contradicts the general position condition of a generic Pach's
  configuration.
\end{proof}

It follows that $\pp_n \in B$ and $\hh_{i,n} \in B$ for all $i$. Since
$\lVert\uu_{i,F,n}\rVert=1$ for every $i$ and $F$, the whole sequence $\xx_n$ is contained in a
compact subset of $\R^{(d+1)2^d+1}$, and so it has a convergent
subsequence $\xx_{n_k}$ with a limit $\xx$.  Let $\hh_i\mathrel{\mathop:}=\lim_k \hh_{i,n_k}$
for every $i\in[d+1]$, $\pp\mathrel{\mathop:}=\lim_k \pp_{n_k}$, and $\uu_{i,F}\mathrel{\mathop:}=\lim_k
\uu_{i,F,n_k}$ for every $i\in[d+1]$ and $F\subseteq [d+1]$, $i\in F$, $F \neq
\{i\}$.

For every $i\in[d+1]$, the point $\hh_i$ and the vectors $\uu_{i,F}$ determine
a (possibly degenerate) convex cone $Z_i$ as follows:
\[
  Z_i\mathrel{\mathop:}=\conv\left(\bigcup_{F\subseteq [d+1], i\in F, F\neq\{i\}} \{\hh_i+\lambda \uu_{i,F}, \; \lambda\ge 0\}\right).
\] 

Note that if $\hh_1, \dots, \hh_{d+1}$ are affinely independent, and thus form a nondegenerate simplex, then
the cones $Z_i$ correspond to the corner regions $C_i$ defined in Section~\ref{s:corner}, and are limits of the regions $Z_{i,n_k}$, in a certain sense that we define shortly. However, if $\hh_1, \dots, \hh_{d+1}$ span a subspace of dimension at most $d-1$ (some of the points may even coincide), these points alone do not provide enough information to reconstruct the cones $Z_i$. In particular, if $\hh_i$ is in the convex hull of the vertices $\hh_j$ for $j\in[d+1]\setminus\{i\}$, we need some of the vectors $\uu_{i,F}$, too.

We create an ``epsilon of room''~\cite{Tao10_epsilon_I} around $Z_i$.
For every $m\in \mathbb{N}$, we define a neighborhood $U_{i,m}$ of $Z_i$ as an
infinite union of (possibly nonconvex) open cones whose apices $\aa$ are close to
$\hh_i$ and whose rays have directions close to the directions of the
rays of $Z_i$ (see Figure~\ref{fig_epsilon_room}): 
\[
U_{i,m}\mathrel{\mathop:}=\left\{\aa+\ww; \; \lVert\aa-\hh_i \rVert < 1/m, \, \ww \neq 0, \text{ and } \dist\left(\hh_i + \frac{\ww}{\lVert\ww \rVert}, Z_i\right) < 1/m\right\}.
\]
We show that for each $i\in [d+1]$, the cone $Z_i$ is a limit of the cones $Z_{i,n_k}$ in the following sense.
\begin{claim}\label{claim_limita_kuzelu}
\begin{enumerate}
\item[{\rm (1)}] $Z_i$ is in the pointwise limit of $Z_{i,n_k}$. That is, for every $\zz \in Z_i$, there is a sequence of points $\zz_k \in Z_{i,n_k}$ converging to $\zz$.
\item[{\rm (2)}] 
$Z_i$ is an intersection of the sequence of open neighborhoods $U_{i,m}$ and for every $m$, if $\xx_{n_k}$ is sufficiently close to $\xx$ then $Z_{i,n_k}\subset U_{i,m}$.
\end{enumerate}
\end{claim}

Part (1) of Claim~\ref{claim_limita_kuzelu} follows directly from the
definition of $Z_i$ and from the fact that $Z_{i,n_k}$ is a convex hull of rays
$\{\hh_{i,n_k}+\lambda \uu_{i,F,n_k}, \, \lambda\ge 0\}$ for $F\subseteq [d+1],
i\in F, F \neq \{i\}$, and these rays pointwise converge to the rays $\{\hh_i+\lambda \uu_{i,F}, \, \lambda\ge 0\}$. 

It is also clear that $Z_i=\bigcap_{m=1}^{\infty} U_{i,m}$. To establish the rest of part (2) of the claim, we need the full data from the definition of $Z_i$. Since the proof is rather technical, we delegate it into Subsection~\ref{sub_limita_kuzelu}.

We set $F_{i,m}$ as the closure of $U_{i,m}$. Clearly, $F_{i,m+1} \subset U_{i,m}$ for every $m$. By Lemma~\ref{lemma_outer_regular}, 
\begin{equation}\label{eq_prunik_otevrenych}
\mu_i(Z_i)=\inf_{m\in \mathbb{N}} \mu_i(U_{i,m})=\inf_{m\in \mathbb{N}} \mu_i(F_{i,m}). 
\end{equation}
By Claim~\ref{claim_limita_kuzelu}(1) and since $\pp_{n_k}$ is contained in all $(Z_{1,n_k} \dots, Z_{d+1,n_k})$-simplices by Lemma~\ref{l:in_simplex}, we conclude that the point $\pp$ is contained in all $(Z_1, \dots, Z_{d+1})$-simplices. Note that $\pp$ is not necessarily in the interior of these simplices; moreover, the simplices may be degenerate.

\begin{figure}
\begin{center}
\epsfig{file={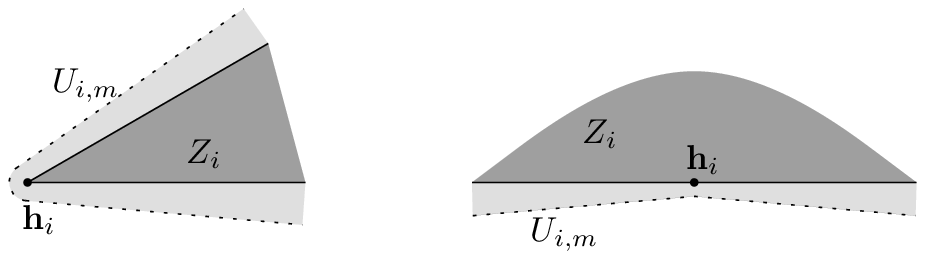}}
\end{center}
\caption{A neighborhood $U_{i,m}$ of a cone $Z_i$. The neighborhood is nonconvex if $\hh_i$ is not an extreme point of $Z_i$ (right)}
\label{fig_epsilon_room}
\end{figure}

It remains to show that $\mu_i(Z_i)\ge \gamma(d)$ for every $i$.
Fix $i\in [d+1]$ and let $\varepsilon >0$. 
By~(\ref{eq_prunik_otevrenych}), there is an $m\in \mathbb{N}$ such that 
\begin{equation}\label{eq_vnejsi_regularita}
\mu_i(Z_i) > \mu_i(F_{i,m}) - \varepsilon.
\end{equation}
By Theorem~\ref{theorem_weak_closed}, there is a $k_0$ such that for all $k>k_0$, we have
\begin{equation}\label{eq_weak_conv}
\ \mu_i(F_{i,m}) > \mu''_{i,n_k}(F_{i,m}) - \varepsilon.
\end{equation}
By Claim~\ref{claim_limita_kuzelu}(2), there is a $k>k_0$ such that $Z_{i,n_k}\subset U_{i,m} \subset F_{i,m}$, and therefore
\begin{equation}\label{eq_containment}
\mu''_{i,n_k}(F_{i,m}) \ge \mu''_{i,n_k}(Z_{i,n_k}).
\end{equation}
Combining~(\ref{eq_vnejsi_regularita}),~(\ref{eq_weak_conv}) and~(\ref{eq_containment}) with the assumption $\mu''_{i,n}(Y_{i,n})\ge \gamma(d)$ we obtain
\[
\mu_i(Z_i) > \mu''_{i,n}(Z_{i,n}) -2\varepsilon \ge \mu''_{i,n}(Y_{i,n})-2\varepsilon \ge \gamma(d)-2\varepsilon. 
\]
Since the $\varepsilon$ can be taken arbitrarily small, the theorem follows.

\subsection{Proof of Claim~\ref{claim_limita_kuzelu}(2)}\label{sub_limita_kuzelu}

Given a point $\aa \in \R^d$ and $d$ unit vectors $\vv_1, \dots, \vv_d \in \R^d$, the
\emph{cone} with \emph{apex} $\aa$ \emph{induced by} $\vv_1, \dots, \vv_d$ is
defined as
\[
\CC(\aa,\vv_1, \dots, \vv_d) \mathrel{\mathop:}= \conv(\{\aa + \lambda\vv_i; \; \lambda \ge 0, \, i \in [m]\}).
\]
If the vectors $\vv_1, \dots, \vv_d \in \R^d$ are linearly independent, the cone $\CC(\aa,\vv_1, \dots, \vv_d)$ is also called \emph{simplicial}, and the rays $\{\aa + \lambda\vv_i, \, \lambda \ge 0\}$, for $i\in [d]$, are called the \emph{extreme rays} of $\CC(\aa,\vv_1, \dots, \vv_d)$.

As we have already observed, the cone $Z_{i,n}$ is a simplicial cone with apex
$\hh_{i,n}$ and is induced by the $d$ vectors $\vv_{i,j,n}$,
$j\in[d+1]\setminus\{i\}$. When attempting to define the limit of the sequence
$Z_{i,n_k}$, a difficulty arises when the maximum angle between pairs of rays
in $Z_{i,n_k}$ approaches $\pi$; see Figure~\ref{f:converge}. For $d\ge 3$, this may happen even if all
pairs of extreme rays form an angle at most $2\pi/3$, or, in general,
$\arccos(-1/(d-1))$. We have introduced the vectors $\uu_{i,F,n}$ to remedy
this difficulty. 

\begin{figure}
  \begin{center}
    \includegraphics{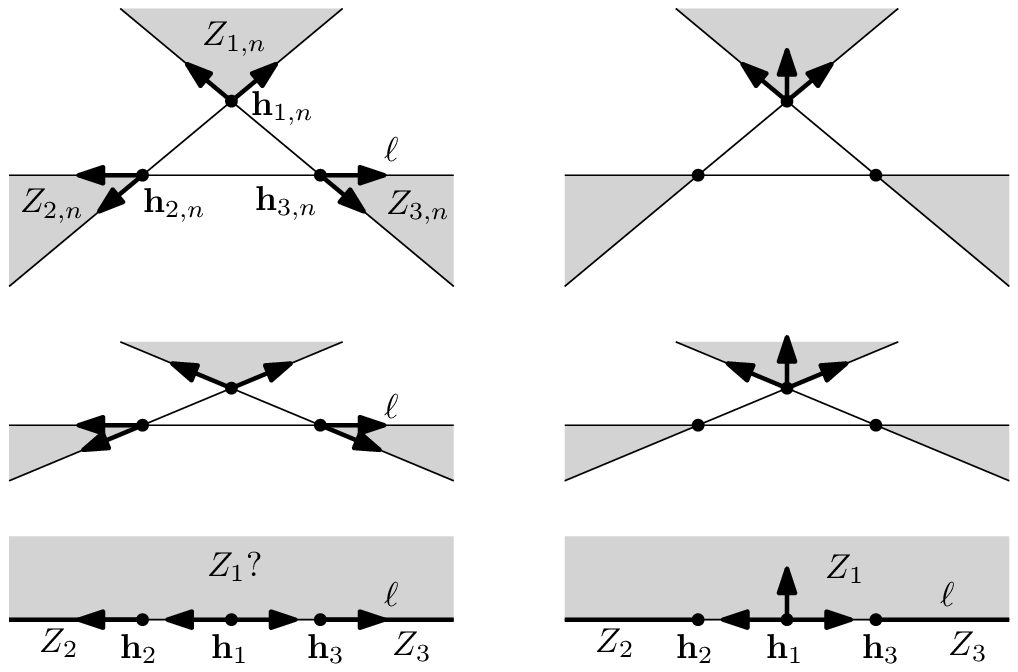}
  \caption{The vectors $\vv_{i,j,n}$ provide a useful information about the
  convex cone $Z_i$; however, they are not sufficient in general to determine
$Z_i$. Let us consider the case when $\hh_2 \mathrel{\mathop:}= \hh_{2,n}$ and $\hh_3 \mathrel{\mathop:}= \hh_{3,n}$ are
fixed points of a line $\ell$ and $\hh_{1,n}$ approaches the midpoint $\hh_1$ of
$\hh_2\hh_3$ from above; see the series of pictures on the left. The vectors
$\vv_{i,j,n}$ are drawn as small arrows without labels.
Then the (limit) vectors $\vv_{2,j}$ determine the (limit) cone $Z_2$, which is a ray in this case.
However, the two vectors $\vv_{1,j}$ (around $\hh_1$ on the bottom left
picture) are insufficient to determine the expected limit
cone $Z_1$. By using all the vectors $\uu_{1,F,n}$, we can determine the cone $Z_1$ as depicted on the right}
 \label{f:converge}
  \end{center}
\end{figure}

\begin{observation}\label{obs_ostre}
Suppose that $i\in F\subset K\subseteq [d+1]$. Then $\uu_{i,F,n}\cdot\uu_{i,K,n}>0$. That is, $\uu_{i,F,n}$ and $\uu_{i,K,n}$ form an angle smaller than $\pi/2$.
\end{observation}

\begin{proof} By condition a) for the vector $\uu_{i,F,n}$, the vector $\uu_{i,F,n}$ is a nonnegative linear combination of the vectors $\vv_{i,j,n}$ for $j\in F \setminus\{i\}$. By condition b) for the vector $\uu_{i,K,n}$, we have $\uu_{i,K,n}\cdot \vv_{i,j,n} >0$ for every $j\in K \setminus\{i\}$. The observation follows by combining these inequalities.
\end{proof}

An \emph{$i$-chain} is a sequence $\mathcal{F}=(F_0,F_1, F_2, \dots,F_d)$ of
nonempty subsets of $[d+1]$ such that $\{i\}=F_0\subset F_1\subset F_2 \subset
\cdots \subset F_d=[d+1]$. In particular, $\lvert F_j \rvert = j+1$ for every
$j\in [d]$. For every $n$, every $i$-chain $\mathcal{F}$ determines a
simplicial cone $C(\mathcal{F},n) \mathrel{\mathop:}=
C(\hh_{i,n},\uu_{i,F_1,n}, \uu_{i,F_2,n}, \dots, \uu_{i,F_d,n})$. It is easy to
see that for fixed $n$ and $i$, the cones $C(\mathcal{F},n)$ determined by all
$i$-chains $\mathcal{F}$ cover $Z_{i,n}$; we explain this in more detail below. In fact, their interiors are also pairwise disjoint so they form a finite \emph{tiling} of $Z_{i,n}$. 
It is therefore sufficient to prove the conclusion of
Claim~\ref{claim_limita_kuzelu}(2) for each sequence of cones $C(\mathcal{F},n_k)$ separately.

In order to show that the cones $C(\mathcal{F},n)$ cover $Z_{i,n}$, it is
sufficient to show that $C(\mathcal{F},n) \cap H$ cover $Z_{i,n} \cap H$ for
any hyperplane $H$ perpendicular to the vector $\uu_{i,[d+1],n}$ and such that
$\lvert Z_{i,n} \cap H\rvert > 1$. Then $\Delta \mathrel{\mathop:}= 
H \cap Z_{i,n}$ is a $(d-1)$-simplex meeting all the rays $\{\lambda \uu_{i,F,n}, \, \lambda \ge 0\}$, in points $\rr_F$ (considering $i$ and $n$ as fixed).
Similarly $\Delta(\F) \mathrel{\mathop:}= H \cap C(\F,n)$ is a $(d-1)$-simplex. Its vertices are
the points $\rr_F$ for $F$ belonging to $\F$. See Figure~\ref{f:tiling}. Each
$\rr_F$ is in the relative interior of some face $\Delta_F$ of $\Delta$. If we were
lucky and $\rr_F$ would coincide, for each $F\subseteq [d+1]$, $i \in F$, $F \neq
\{i\}$, with the barycenter $\bb_F$ of $\Delta_F$, then it is well known that the
simplices $\Delta(F)$ form the barycentric subdivision of $\Delta$ and
therefore they tile $\Delta$. If this is not the case, we consider the
piecewise-linear homeomorphism of $\Delta$, linear on each $\Delta(\F)$,
sending each $\rr_F$ to $\bb_F$. This again shows that $\Delta(\F)$ tile
$\Delta$ because a homeomorphism maps a tiling to a tiling.

\begin{figure}
  \begin{center}
  \includegraphics{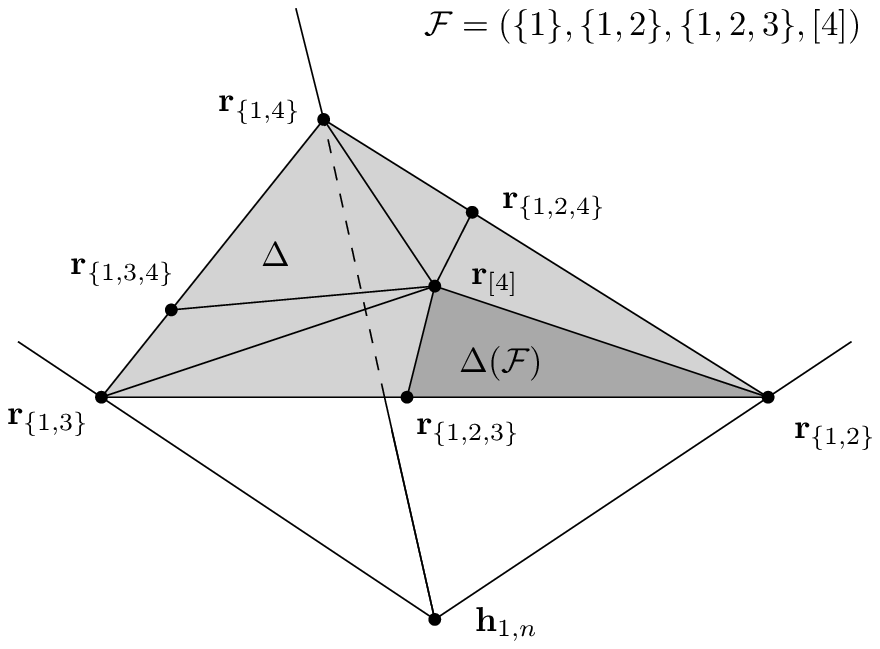}
  \caption{Tiling the cone}
  \label{f:tiling}
\end{center}
\end{figure}

A simplicial cone $\CC(\aa,\vv_1, \dots, \vv_d)$ is \emph{acute} if $\vv_i \cdot
\vv_j > 0$ for any $i, j \in [d]$. Observe that in an acute simplicial cone, every two (not necessarily extreme) rays form an acute angle.
Observation~\ref{obs_ostre} implies that every cone $C(\mathcal{F},n)$ is acute.

An \emph{admissible} vector of a cone $\CC \mathrel{\mathop:}=
\CC(\aa,\vv_1,\dots,\vv_d)$ is a unit vector $\vv$ such that $\aa + \vv \in
\CC$. That is, admissible vectors form an intersection of the unit sphere with $C-\aa$. 
Another equivalent definition is that $\vv$ is admissible if it can be written as 
\[
\vv = \frac{\lambda_1 \vv_1 + \cdots + \lambda_d \vv_d}{\lVert\lambda_1 \vv_1 + \cdots
+ \lambda_d \vv_d\rVert}
\]
where $\lambda_i \geq 0$ for $i \in [d]$ and at least one of these $\lambda_i$ is
strictly positive. Since this definition is not affected by multiplying
each $\lambda_i$ by a positive constant, we can further require that $\lambda_1 + \cdots +
\lambda_d = 1$.

Let $\CC^n \mathrel{\mathop:}= \CC(\aa^n,\vv_1^n, \dots, \vv_d^n)$ be a sequence of simplicial cones
such that the sequence $(\aa^n, \vv_1^n, \dots, \vv_d^n)$ converges to a point 
$(\aa, \vv_1, \dots, \vv_d)$. Then the cone $\CC \mathrel{\mathop:}= \CC(\aa,\allowbreak \vv_1, \dots,\allowbreak
\vv_d)$ is the \emph{limit} of $\CC^n$. Our aim is to show that the limit of acute cones
behaves nicely. Clearly, if the cones $\CC^n$ are acute, then $\vv_i\cdot \vv_j\ge 0$ for any $i,j\in[d]$.
 Claim~\ref{claim_limita_kuzelu}(2) now follows from the following claim,
 applied to every sequence $C(\mathcal{F},n_k)$.

\begin{claim}\label{claim_konvergence}
Let $\CC^n \mathrel{\mathop:}= \CC(\aa^n,\vv_1^n, \dots, \vv_d^n)$ be a sequence of acute 
simplicial cones with a limit $\CC \mathrel{\mathop:}= \CC(\aa,\vv_1, \dots, \vv_d)$.
Then for every $\varepsilon > 0$ there is an $n_0 \in \N$ such that for every $n
\geq n_0$ and every admissible vector $\vv^n$ of $\CC^n$, there is an admissible
vector $\vv$ of $\CC$ such that $\lVert \vv - \vv^n\rVert \leq \varepsilon$.

\end{claim}

\begin{proof}

  For a given $\varepsilon$ we set $\delta \mathrel{\mathop:}= 
  \varepsilon/(2d^2)$ and we
  choose such $n_0$ that for every $n\ge n_0$ and for every $i \in [d]$, we have $\lVert\vv_i - \vv_i^n\rVert<\delta$.

Let $\vv^n$ be an admissible vector of $\CC^n$ written as 
\[
\vv^n = \frac{\lambda_1 \vv^n_1 + \cdots + \lambda_d \vv^n_d}{\lVert\lambda_1 \vv^n_1 +
\cdots + \lambda_d \vv^n_d\rVert}
\]
with $\lambda_1 + \cdots + \lambda_d=1$. Then
\[
\vv \mathrel{\mathop:}= \frac{\lambda_1 \vv_1 + \cdots + \lambda_d \vv_d}{\lVert\lambda_1 \vv_1 + \cdots
+ \lambda_d \vv_d\rVert}
\]
is an admissible vector of $\CC$.
Our aim is to show that $\lVert\vv - \vv^n\rVert$ is small.
Let $\xx \mathrel{\mathop:}= \lambda_1 \vv_1 + \cdots + \lambda_d \vv_d$ and 
$\xx_n \mathrel{\mathop:}= \lambda_1 \vv^n_1 + \cdots + \lambda_d \vv^n_d$.

By the triangle inequality, we have 
\[
\Bigl\lvert  \lVert\xx\rVert - \lVert\xx^n\rVert  \Bigr\rvert   \le   \lVert\xx - \xx^n\rVert
\le \lambda_1 \lVert\vv_1 - \vv^n_1\rVert + \cdots + \lambda_d\lVert\vv_d - \vv^n_d\rVert \le d\delta.
\]

Clearly, $\lVert\xx\rVert \le 1$, since $\lVert\xx\rVert$ is a convex combination of unit vectors.
We further prove that $\lVert\xx\rVert \ge 1/\sqrt{d}$. Analogously, we also get the inequality $\lVert\xx^n\rVert \ge 1/\sqrt{d}$. We have
\[
\lVert\xx\rVert^2 = \xx \cdot \xx = \lambda_1^2 + \cdots + \lambda_d^2 + \sum\limits_{1
\leq i < j \leq d} 2 \lambda_i \lambda_j (\vv_i \cdot \vv_j) \geq \lambda_1^2 +
\cdots + \lambda_d^2 \geq \frac{(\lambda_1 + \cdots + \lambda_d)^2}d = \frac{1}{d},
\]
using the observation that $\vv_i\cdot \vv_j\ge 0$ and the inequality of
arithmetic and quadratic means.

Finally, 

\begin{align*}
\lVert\vv - \vv^n\rVert &= \left\lVert \frac{\xx}{\lVert\xx\rVert} - \frac{\xx^n}{\lVert\xx^n\rVert} \right\rVert
\le  d \cdot \Bigl\lVert         \lVert\xx^n\rVert \cdot \xx  -  \lVert\xx\rVert \cdot \xx^n  \Bigr\rVert \\
&\le d \cdot \left( \Bigl\lVert  \lVert\xx^n\rVert \cdot \xx  -  \lVert\xx\rVert \cdot \xx    \Bigr\rVert + 
                    \Bigl\lVert  \lVert\xx\rVert \cdot \xx    -  \lVert\xx\rVert \cdot \xx^n  \Bigr\rVert \right) \\
&= d \cdot \left( \Bigl\lvert \lVert\xx\rVert  -  \lVert\xx^n\rVert \Bigr\rvert  \cdot  \lVert\xx\rVert 
                 +  \lVert\xx\rVert   \cdot  \lVert\xx - \xx^n\rVert  \right) \\
 &\le 2d^2 \delta = \varepsilon.
\end{align*}

\end{proof}

\subsection{Final remark}
We note that any future improvement of Theorem~\ref{t:bounds}(2) yields a corresponding improvement of Theorem~\ref{t:pach_continuous}. To see this, we have to modify the proof of Theorem~\ref{t:pach_continuous} a little bit, since we cannot rely on the proof of Theorem~\ref{t:bounds}(2) to obtain the arrangement $\H_n$ satisfying all the required conditions. 
Instead, we use Lemmas~\ref{l:shrinkY} and~\ref{lemma_generic_strict_separation}. Also, when choosing the sequence of measures $\mu''_{i,n}$, we require, in addition, that $\supp(\mu''_{1,n}) \cup \dots \cup \supp(\mu''_{d+1,n})$ satisfies condition (G) and that $|\supp(\mu''_{i,n})|\ge n$, which will compensate for the loss of some points after applying Lemma~\ref{l:shrinkY}.

\section*{Acknowledgement}
We thank Erik Aas for participating at the initial stage of this project.
We also thank 
Karim Adiprasito for fruitful discussions about the minimum solid angle in a
simplex and 
Andrew Suk for a brief discussion about the
expected improved lower bound by Fox, Pach, and Suk.

R.~K. was supported by the Russian Foundation for Basic Research grant
15-31-20403 (mol\_a\_ved).
J.~K., Z.~P. and M.~T. were partially supported by ERC Advanced Research Grant
no 267165 (DISCONV) and by the project CE-ITI (GA\v{C}R P202/12/G061) of the
Czech Science Foundation.
J.~K. was also partially supported by Swiss National Science Foundation Grants 200021-137574 and 200020-14453.
P.~P., Z.~P. and M.~T. were partially supported by the Charles University Grant GAUK 421511.
P.~P. was also partially supported by the Charles University Grant SVV-2014-260107.
Z.~P. was also partially supported by the Charles University Grant SVV-2014-260103.
Part of this work was done when M.~T. was affiliated with Institutionen f\"{o}r matematik, Kungliga Tekniska
H\"{o}gskolan, 100~44 Stockholm.

\bibliographystyle{alpha}
\bibliography{pachbib}

\end{document}